\RequirePackage{fix-cm}

\documentclass[smallextended,envcountsect,envcountsame]{svjour3}       
\usepackage{mathptmx}      

\usepackage{latexsym,amsmath,amsfonts,amscd,amssymb,verbatim}

\begin{document}

\title{Generalized Polyhedral Convex Optimization Problems}


\author{Nguyen Ngoc Luan \and Jen-Chih Yao}


\institute{Nguyen Ngoc Luan \at Department of Mathematics and Informatics, Hanoi National University of Education, 136 Xuan Thuy, Hanoi, Vietnam. This author was supported by Hanoi National University of Education.\\
\email{luannn@hnue.edu.vn}           
\and Jen-Chih Yao  \at Center for General Education, China Medical University, 	Taichung 40402, Taiwan. This author was partially supported by the Grant MOST 105-2115-M-039-002-MY3.\\
\email{yaojc@mail.cmu.edu.tw}}

\date{Received: date / Accepted: date}

\maketitle

\begin{abstract}
Generalized polyhedral convex optimization problems in locally convex Hausdorff topological vector spaces are studied systematically in this paper. We establish solution existence theorems, necessary and sufficient optimality conditions, weak and strong duality theorems. In particular, we show that the dual problem has the same structure as the primal problem, and the strong duality relation holds under three different sets of conditions. 



\end{abstract}

\section{Introduction}

A \textit{polyhedral convex set} in a finite-dimensional Euclidean space is the intersection of a finite number of closed half-spaces; see, e.g., \cite[Section~19]{Rockafellar_1970}. Functions with polyhedral convex epigraphs, called \textit{polyhedral convex functions}, were investigated long time ago by Rockafellar \cite{Rockafellar_1970}. Later, Rockafellar and Wets \cite[p.~68]{Rockafellar_Wets_1998} showed that a polyhedral convex function can be characterized as the maximum of a finite family of affine functions over a certain polyhedral convex set. A minimization problem is said to be a \textit{polyhedral convex optimization problem} if the objective function is polyhedral convex and the constraint set is also polyhedral convex. The concepts of polyhedral convex function and polyhedral convex optimization problem have attracted much attention from researchers (see Rockafellar and Wets \cite{Rockafellar_Wets_1998}, Bertsekas, Ned{\'{i}}c, and Ozdaglar \cite{Bertsekas_et_al_2003}, Boyd and Vandenberghe \cite{Boyd_Vandenberghe_2004}, Bertsekas \cite{Bertsekas_2009,Bertsekas_2015}, and the references therein).

\medskip
The definition of generalized polyhedral convex set was proposed by Bonnans and Shapiro \cite[Definition~2.195]{Bonnans_Shapiro_2000}. A subset of a locally convex Hausdorff topological vector space is said to be a \textit{generalized polyhedral convex set} (gpcs) if it is the intersection of finitely many closed half-spaces and a closed affine subspace of that topological vector space. When the affine subspace can be chosen as the whole space, the generalized polyhedral convex set is called a \textit{polyhedral convex set} (pcs), or a convex polyhedron. Clearly, in a finite-dimensional space, a subset is generalized polyhedral convex if and only if it is polyhedral convex. We observe that in any infinite-dimensional space, every nonempty polyhedral convex set is unbounded (see \cite[Lemma~2.12]{Luan_Yao_Yen_2016} for details). Hence, the notion of generalized polyhedral convex set appears naturally in the case where the spaces under consideration are infinite-dimensional. The theories of generalized linear programming and quadratic programming in \cite[Sections~2.5.7 and 3.4.3]{Bonnans_Shapiro_2000} are based on the concept of generalized polyhedral convex set. In a Banach space setting, various applications of gpcs can be found in the papers by Ban, Mordukhovich and Song~\cite{Ban_Mordukhovich_Song_2011},  Gfrerer~\cite{Gfrerer_2013_SIOPT,Gfrerer_2014_SVA}, Ban and Song~\cite{Ban_Song_2016}.

\medskip
In a locally convex Hausdorff topological vector space setting, by using a representation formula for generalized polyhedral convex sets, Luan and Yen~\cite{Luan_Yen_2015} have obtained solution existence theorems for generalized linear programming problems, a scalarization formula for the weakly efficient solution set of a generalized linear vector optimization problem, and proved that the latter is the union of finitely many generalized polyhedral convex sets. In \cite{Luan_2016}, where the relative interior of the dual cone of a polyhedral convex cone is described, it is proved that the corresponding efficient solution set is the union of finitely many generalized polyhedral convex sets. Moreover, it is shown that both solution sets of a generalized linear vector optimization problem are connected by line segments. This result extends a classical theorem due to Arrow, Barankin, and Blackwell (see, e.g., \cite{ABB_1953,Luc_1989,Luc_2016}).

\medskip
The recent paper of Luan, Yao, and Yen \cite{Luan_Yao_Yen_2016} can be seen as a comprehensive study on generalized polyhedral convex sets, generalized polyhedral convex functions on locally convex Hausdorff topological vector spaces, and the related constructions such as sum of sets, sum of functions, directional derivative, infimal convolution, normal cone, conjugate function, subdifferential. Among other things, the authors have showed that, under a mild condition, a generalized polyhedral convex set can be characterized by the finiteness of the number of its faces.  

\medskip
It is well known that any infinite-dimensional normed space equipped with the {\it weak topology} is not metrizable, but it is a locally convex Hausdorff topological vector space. Similarly, the dual space of any infinite-dimensional normed space equipped with the {\it weak$^*$-topology} is not metrizable, but it is a locally convex Hausdorff topological vector space. The just mentioned two fundamental models in functional analysis are the most typical examples of locally convex Hausdorff topological vector space, whose topologies cannot be given by norms. 

\medskip
The aim of the present paper is to study the concept of generalized polyhedral convex optimization problems in locally convex Hausdorff topological vector spaces. Our investigation is based on the above-mentioned papers of Luan and Yen~\cite{Luan_Yen_2015}, Luan, Yao, and Yen \cite{Luan_Yao_Yen_2016}.

\medskip
The remaining part of our paper has four sections. Section 2 collects some necessary preliminaries. Section 3 is devoted to the solution existence of generalized polyhedral convex optimization problems. Optimality conditions for generalized polyhedral convex optimization problem are studied in Section 4. A duality theory for this class of  problems is presented in the final section.  

\section{Preliminaries}
\setcounter{equation}{0}

In the sequel, we will need some results on generalized polyhedral convex sets and generalized polyhedral convex set functions, which are recalled below. 

\medskip
From now on, if not otherwise stated, $X$ is a {\it locally convex Hausdorff topological vector space} (lcHtvs). Denote by $X^*$ the dual space of $X$ and by  $\langle x^*, x \rangle$ the value of~$x^* \in X^*$ at $x \in X$. The \textit{annihilator} \cite[p.~117]{Luenberger_1969} of a subset $C\subset X$, denoted by $C^{\perp}$, is defined by $C^{\perp}:=\{x^* \in X^* \mid \langle x^*, u \rangle = 0, \ \forall u \in C\}$. For a subset $\Omega$ of $X$, by $\overline{\Omega}$ we denote the topological closure of $\Omega$. This notation is also used for subsets of $X^*$.        

\begin{definition}{\rm (See \cite[p.~133]{Bonnans_Shapiro_2000})} A subset $D \subset X$ is said to be a \textit{generalized polyhedral convex set}, or a \textit{generalized convex polyhedron}, if there exist $x^*_i \in X^*$, $\alpha_i \in \mathbb R$, $i=~1,2,\dots,p$, and a closed affine subspace $L \subset X$, such that 
\begin{equation}\label{eq_def_gpcs}
D=\left\{ x \in X \mid x \in L,\ \langle x^*_i, x \rangle \leq \alpha_i,\  i=1,\dots,p\right\}.
\end{equation} 
If $D$ can be represented in the form \eqref{eq_def_gpcs} with $L=X$, then we say that it is a \textit{polyhedral convex set}, or a \textit{convex polyhedron}.
\end{definition}

If $L \subset X$ is a closed affine subspace, then one can find a continuous surjective linear mapping $A$ from $X$ to a lcHtvs $Y$ and a vector $y \in Y$ such that $L=\left\{x \in X \mid Ax=y  \right\}$ (see \cite[Remark~2.196]{Bonnans_Shapiro_2000}). Therefore, we can rewrite \eqref{eq_def_gpcs} in the form
\begin{equation}\label{eq_def_gpcs_2}
D=\left\{ x \in X \mid Ax=y,\  \langle x^*_i, x \rangle \leq \alpha_i,\  i=1,\dots,p\right\}.
\end{equation}   
It is clear that, when $X$ is finite-dimensional, a subset $D\subset X$ is a gpcs if and only if it is a pcs. 

\medskip
Later on, if not otherwise stated, $D \subset X$ is a nonempty generalized polyhedral convex set given by \eqref{eq_def_gpcs_2}. Set $I=\{1,\dots,p\}$ and $I(x)=\left\{i \in I \mid \langle x^*_i, x \rangle = \alpha_i \right\}$ for $x \in D$. If $D$ is a pcs, then one can choose $Y=\{0\}$, $A\equiv 0$, and $y=0$. 

\medskip
Let $C \subset X$ be a nonempty convex set. As in \cite[p.~61]{Rockafellar_1970}, the {\it recession cone} of $C$ is defined by $$0^{+}C:=\left\{v \in X \mid x+tv \in C, \ \forall x \in X,\ \forall t \geq 0\right\}.$$ On account of \cite[p.~33]{Bonnans_Shapiro_2000}, if $C$ is nonempty and closed, then $0^+C$ is a closed convex cone, and $v \in X$ belongs to~$0^{+}C$ if and only if there exists an $x \in C$ such that $x + tv \in C$ for all $t \geq 0$. 

\begin{remark}\label{rec_cone_gpcs} If a nonempty generalized polyhedral convex set $D$ is given by \eqref{eq_def_gpcs_2}, then its recession cone can be computed by the formula $$0^{+}D=\left\{ v \in X \mid Av=0,\  \langle x^*_i, v \rangle \leq 0,\  i=1,\dots,p\right\}.$$ It follows that $0^{+}D$ is a generalized polyhedral convex cone.
\end{remark}

Following \cite[p.~122]{Aubin_Frankowska_2009}, we can define the Bouligand-Severi \textit{tangent cone} $T_C(x)$ to a closed subset $C \subset X$ at $x \in C$ as the set of all $v\in X$ such that there exist sequences $t_k\to 0^+$ and $v_k\to v$ such that $x+t_kv_k\in C$ for every $k$. If $C$ is a nonempty convex set, then $T_C(x)=\overline{{\rm cone}\,(C-x)}.$ By \cite[Proposition~2.19]{Luan_Yao_Yen_2016}, if $C$ is a generalized polyhedral convex set (resp., a polyhedral convex set) then, for any $x\in C$, the cone $T_C(x)$ is generalized polyhedral convex (resp., polyhedral convex) and one has $T_C(x)={\rm cone}\,(C-x)$.

\medskip
Let $f$ be a function from $X$ to $\bar{\mathbb{R}}:=\mathbb{R} \cup \{\pm \infty\}$. The \textit{effective domain} and the \textit{epigraph} of $f$ are defined respectively by setting ${\rm dom}\,f=\{x \in X \mid f(x) < +\infty\}$ and ${\rm epi}\,f=\left\{(x, \alpha) \in X \times \mathbb{R} \mid x \in {\rm dom}\,f,\ f(x) \leq \alpha \right\}.$  If ${\rm dom}\,f$ is a nonempty set and $f(x) > - \infty$ for all $x \in X$, then $f$ is said to be \textit{proper}. One says that $f$ is  \textit{convex} if  ${\rm epi}f$ is a convex set in $X \times \mathbb{R}$. 

\medskip
Following \cite[p.~66]{Rockafellar_1970}, we define the {\it recession function} $f0^{+}$ of a proper convex function $f: X \to \bar{\mathbb{R}}$ by the formula
\begin{equation}\label{eq_rec_fun}
f0^{+}(v)=\inf \left\{\mu \in \mathbb{R} \mid (v, \mu) \in 0^{+}({\rm epi} f) \right\} \quad (v \in X).
\end{equation}

\begin{remark} If $f$ is nonconvex, similar notions bearing the names of \textit{asymptotic function} \cite[p.~48]{Auslender_Teboulle} and \textit{horizon function} \cite[p.~86]{Rockafellar_Wets_1998}  have been defined. It is not difficult to show that if $f$ is proper convex and lower semicontinuous (i.e., ${\rm epi}f$ is a closed convex set), these notions coincide with that of recession function. 
\end{remark}

\begin{definition}\label{Def_gpcf} {\rm (See \cite[Definition~3.1]{Luan_Yao_Yen_2016})} One calls $f:X\to\bar{\mathbb{R}}$ is a \textit{generalized polyhedral convex function} (resp., a \textit{polyhedral convex function}) if the epigraph ${\rm epi}f$ is a generalized polyhedral convex set (resp., a polyhedral convex set) in $X\times\mathbb R$. If $-f$ is a generalized polyhedral convex function (resp., a \textit{polyhedral convex function}), then~$f$ is said to be a \textit{generalized polyhedral concave function} (resp., a \textit{polyhedral concave function}).
\end{definition}

From Definition~\ref{Def_gpcf}, we can assert that every generalized polyhedral convex function is a convex function. Of course, in the case where $X$ is finite-dimensional, a function $f:X\to\bar{\mathbb{R}}$ is generalized polyhedral convex if and only if it is polyhedral convex. 

\medskip
The following theorem shows that, any generalized polyhedral convex function (resp., any polyhedral convex function) can be represented in the form of the maximum of a finite family of continuous affine functions over a certain generalized polyhedral convex set (resp., a polyhedral convex set). 

\begin{theorem}\label{properties_gpcf}{\rm (See \cite[Theorem 3.2]{Luan_Yao_Yen_2016})} The following properties of a proper convex function $f:X\to\bar{\mathbb{R}}$ are equivalent:
	\begin{description}
		\item{\rm (a)} \textit{$f$ is generalized polyhedral convex (resp., polyhedral convex)};
		\item{\rm (b)} \textit{${\rm dom}\,f$ is a generalized polyhedral convex set (resp., a polyhedral convex set) in~$X$ and there exist $v_k^* \in X^*$, $\beta_k \in \mathbb{R}$, for $k=1,\dots,m$, such that}
		\begin{equation}\label{eq_rep_gcpf}
		f(x)=\max \left\{ \langle v_k^*, x \rangle + \beta_k \mid k=1,\dots,m  \right\} \quad (x \in {\rm dom}\,f).
		\end{equation}	  
	\end{description} 
\end{theorem}

Consider a {\it generalized polyhedral convex optimization problem} 
\begin{equation*}
{\rm (}{\mathcal{P}}{\rm )} \qquad \min \left\{ f(x) \mid x \in D\right\} 
\end{equation*}
where, as before, $X$ is a locally convex Hausdorff topological vector space, $D \subset X$ a nonempty generalized polyhedral convex set, and $f: X \to \bar{\mathbb{R}}$ a proper generalized polyhedral convex function. We say that $u \in D$ is a solution of ${\rm (}{\mathcal{P}}{\rm )}$ if~$f(u)$ is finite and $f(u) \leq f(x)$ for all $x \in D$. The solution set of ${\rm (}{\mathcal{P}}{\rm )}$ is denoted by ${\rm Sol}{\rm (}{\mathcal{P}}{\rm )}$.

\medskip
\textit{From now on, if not otherwise stated, the constraint set $D$ is given by \eqref{eq_def_gpcs_2}, and the objective function $f$ is defined by \eqref{eq_rep_gcpf}.}

\medskip
Since  ${\rm dom}\,f$ is a gpcs, it admits the representation
\begin{equation}\label{rep_domf}
{\rm dom}\,f =\left\{ x \in X \mid Bx=z,\  \langle u^*_j, x \rangle \leq \gamma_j,\  j=1,\dots,q\right\},
\end{equation}
where $B$ is a continuous linear mapping from $X$ to a lcHtvs $Z$, $z \in Z$, $u^*_j \in X^*$, $\gamma_j \in \mathbb{ R}$, $j=1,\dots,q$. Set $J=\{1,\dots,q\}$. For each $x \in {\rm dom}\, f$, let $J(x)=\left\{j \in J \mid \langle u^*_j, x \rangle = \gamma_j \right\}$ and $$\Theta(x)=\left\{k \in \{1,\dots,m\} \mid  \langle v_k^*, x \rangle +\beta_k=f(x)\right\}.$$ If $f$ is a polyhedral convex function, then ${\rm dom}\, f$ is polyhedral convex by Theorem~\ref{properties_gpcf}; hence, we can choose $Z=\{0\}$, $B\equiv 0$, and $z=0$.  

\medskip
Let $C \subset X$ be a nonempty convex set. The \textit{normal cone} to $C$ at $x \in C$ is the set $$N_C(x):=\left\{x^* \in X^* \mid \langle x^*, u-x\rangle \leq 0, \ \forall u \in C\right\}.$$ Clearly, $N_C(x)$ is a closed convex cone in~$X^*$, while $C^{\perp}$ is a closed linear subspace of~$X^*$. If $C$ is a linear subspace of $X$, then $N_C(x)=C^{\perp}$ for all $x \in C$. We observe that if $D$ is given by \eqref{eq_def_gpcs_2} then, due to \cite[Proposition~4.2]{Luan_Yao_Yen_2016},
\begin{equation}\label{norm_cone_gpcs}
N_D(x)={\rm cone}\{x^*_i \mid i \in I(x)\}+({\rm ker}\,A)^{\perp}  \quad (x \in D),
\end{equation}
with ${\rm cone}\,\Omega$ denoting the convex cone generated by a subset $\Omega\subset X^*$.

\medskip
As in \cite[p.~172]{Ioffe_Tihomirov_1979}, the {\it conjugate function} $f^{*}: X^* \to \bar{\mathbb{R}}$  of $f:X\to \bar{\mathbb{R}}$ is given by $f^{*}(x^*)=\sup\limits_{x \in X}\, [\langle x^*, x \rangle - f(x)].$ By \cite[ Proposition~3, p.~174]{Ioffe_Tihomirov_1979}, if $f$ is proper convex and lower semicontinuous, then $f^{*}$ is a proper convex lower semicontinuous function. Obviously, $f^{*}(x^*)=\sup\limits_{x \in {\rm dom}\,f}\, [\langle x^*, x \rangle - f(x)]$ for every $x^* \in X^*$. According to \cite[ Theorem~4.12]{Luan_Yao_Yen_2016}, the conjugate function of a proper gpcf is a proper gpcf.   

\medskip
The notion of subdifferential is the basis for optimality conditions and  other issues in convex programming. The \textit{subdifferential} \cite[p.~46]{Ioffe_Tihomirov_1979} of a proper convex function $f$ at~$x \in {\rm dom}f$ is the set
\begin{equation*}
\partial f(x):=\{x^* \in X^* \mid \langle x^*, u - x \rangle \leq f(u)-f(x), \,  \forall u \in X\}.
\end{equation*}
By \cite[Propostion~1, p.~197]{Ioffe_Tihomirov_1979}, an element $x^* \in X^*$ belongs to $\partial f(x)$ if and only if $f(x)+f^*(x^*)=\langle x^*, x \rangle$. If $f$ is a proper generalized polyhedral convex function, then $\partial f(x)$ is a gpcs for every $x \in {\rm dom}\,f$; see \cite[Proposition~4.15]{Luan_Yao_Yen_2016}. For a nonempty convex subset $C \subset X$, we have $\partial \delta(x, C)=N_C(x)$ for any $x \in C$, where~$\delta(\cdot, C)$ is the indicator function of $C$. 

\begin{remark} On account of \cite[Theorem~4.14]{Luan_Yao_Yen_2016}, if $f$ is defined by \eqref{eq_rep_gcpf} with ${\rm dom}\,f$ being given by \eqref{rep_domf} then, for any $x \in {\rm dom}\,f$, we have
	\begin{equation}\label{rep_subd}
	\partial f(x)={\rm conv}\,\left\{v_k^* \mid  k \in \Theta(x)\right\}+{\rm cone}\,\left\{u^*_j \mid  j \in J(x)\right\}+ ({\rm ker}\,B)^{\perp}
	\end{equation}
where ${\rm conv}\,\Omega$ denotes the convex hull of a subset $\Omega\subset X^*$.
\end{remark}	

The specific structure of  generalized polyhedral convex functions allows one to have a subdifferential sum rule without any assumption on continuity.
\begin{lemma}\label{sum_rules} {\rm (See \cite[Theorems 4.16 and 4.17]{Luan_Yao_Yen_2016})} Suppose that $f_1$ is a proper polyhedral convex function. 
	\begin{description}
	\item{\rm (a)} \textit{If $f_2$ is a proper generalized polyhedral convex function, then} 
		\begin{equation*}\label{eq_sum_subdifferentials_gpcf}
		\partial (f_1+f_2)(x)=\overline{\partial f_1(x)+ \partial f_2(x)}, \quad (x \in ({\rm dom}\,f_1) \cap ({\rm dom}\,f_2)).
		\end{equation*}
	\item{\rm (b)} \textit{If $f_2$ is a proper polyhedral convex function, then} 
		\begin{equation*}\label{eq_sum_subdifferentials_pcf_gpcf}
		\partial (f_1+f_2)(x)=\partial f_1(x)+ \partial f_2(x), \quad (x \in ({\rm dom}\,f_1) \cap ({\rm dom}\,f_2)).
		\end{equation*}
	\end{description} 
\end{lemma}	

\section{Solution Existence Theorems}
\setcounter{equation}{0}

Several solution existence theorems for generalized polyhedral convex optimization problems will be obtained in this section. 

\begin{theorem}\label{FW_exist_thm} {\rm (A Frank--Wolfe-type existence theorem)} If $D \cap {\rm dom}\,f$ is nonempty then, ${\rm (}{\mathcal{P}}{\rm )}$ has a solution if and only if there is a real value $\gamma$ such that $f(x) \geq \gamma$ for every $x \in D$.
\end{theorem}
\begin{proof}The necessity  is obvious. To prove the sufficiency, suppose that there exists a constant $\gamma \in \mathbb{R}$ such that $f(x) \geq \gamma$ for all $x \in D$.  Clearly, $\Phi: X \times \mathbb{R} \rightarrow~\mathbb{R}$, $(x, \alpha) \mapsto \alpha$ for all $(x, \alpha) \in X \times \mathbb{R}$, is a linear mapping. Since ${\rm epi}f \cap (D \times \mathbb{R})$ is a nonempty gpcs in $X \times \mathbb{R}$, by \cite[Proposition~2.1]{Luan_2016} we can assert that $T:=\Phi\left({\rm epi}f \cap (D \times \mathbb{R})\right)$ is a nonempty pcs in $\mathbb{R}$. Hence, $T$ is convex and closed. For every $t \in T$, there exists an $x \in D$ satisfying $(x, t) \in {\rm epi}f$, i.e., $t \geq f(x)$; hence $t \geq f(x) \geq \gamma$. In addition, for every $t' \geq t$, since $(x, t') \in {\rm epi}f \cap  (D \times \mathbb{R})$, one has $t' \in T$. So, we must have $T=[\bar{\gamma},+\infty)$ for some $\bar{\gamma} \geq \gamma$. On one hand, for every $x \in D$, the inclusion $(x, f(x)) \in{\rm epi}f \cap (D \times \mathbb{R})$ yields $f(x) \in T$; hence $f(x) \geq \bar{\gamma}$. On the other hand, since $\bar{\gamma} \in T$, we can find $\bar{x} \in D$ such that $(\bar{x}, \bar{\gamma})  \in {\rm epi}f \cap (D \times \mathbb{R})$. Then we have $\bar{\gamma} \geq f(\bar{x})$ and $f(x) \geq \bar{\gamma} \geq f(\bar{x})$ for every $x \in D$. Thus, $\bar{x}$ is a solution of~${\rm (}{\mathcal{P}}{\rm )}$. $\hfill\Box$
\end{proof}

\begin{remark} Due to the similarity of the formulations of Theorem~\ref{FW_exist_thm} and the solution existence theorem in quadra\-tic programming in \cite[p.~158]{Frank_Wolfe_1956} (see also \cite[Theorem~2.1]{Lee_Tam_Yen_2005}), we call the above result a Frank--Wolfe-type existence theorem in generalized polyhedral convex optimization. If the function $f$ is linear, Theorem~\ref{FW_exist_thm} expresses a recent result in \cite[Theorem~3.3]{Luan_Yen_2015}. For the case $X=\mathbb{R}^n$ and $D=X$, the result in Theorem~\ref{FW_exist_thm} is a known one (see \cite[p.~215]{Bertsekas_et_al_2003}). 
\end{remark}

\begin{theorem}\label{E_exist_thm} {\rm (An Eaves-type existence theorem)} Suppose that $D \cap {\rm dom}\,f$ is non\-empty. Then ${\rm (}{\mathcal{P}}{\rm )}$ has a solution if and only if $f0^{+}(v) \geq 0$ for every $v \in 0^+D$.
\end{theorem}

For proving this theorem, we need a lemma.
\begin{lemma}\label{rep_rec_gpcf} If $f$ is a proper generalized polyhedral convex function given by \eqref{eq_rep_gcpf}, then
	\begin{equation}\label{eq_rec_gpcf}
	f0^{+}(v)=\begin{cases}
	\max \left\{ \langle v_k^*, v \rangle \mid k=1,\dots,m \right\} & \text{ if }  v \in 0^{+}({\rm dom}\,f)\\
	+\infty & \text{ if }  v \notin 0^{+}({\rm dom}\,f).
	\end{cases}
	\end{equation}
In particular, $f0^{+}$ is a proper generalized polyhedral convex function.	
\end{lemma}
\begin{proof} Suppose that ${\rm dom}\, f$ is of the form \eqref{rep_domf}. Then one gets
\begin{equation*}
\begin{aligned}
{\rm epi}f	&=\big\{(x, t) \in X \times \mathbb{R} \mid Bx=z, \  \langle u^*_j, x \rangle \leq \gamma_j, \  j=1,\dots,q, \\
&\hspace*{4cm} \langle v^*_k, x \rangle + \beta_k \leq t,\ k=1,\dots,m \big\}\\
&=\big\{(x, t) \in X \times \mathbb{R} \mid Bx+0t=z,\ \langle u^*_j, x \rangle +0t\leq \gamma_j,\ j=1,\dots,q, \\
&\hspace*{4.7cm}  \langle v^*_k, x \rangle - t \leq -\beta_k ,\ k=1,\dots,m \big\}.\\
\end{aligned}
\end{equation*}
Hence, applying Remark~\ref{rec_cone_gpcs} to ${\rm epi}\, f$ gives
\begin{equation*}
\begin{aligned}
0^{+}({\rm epi}f)&=\big\{(v, \mu) \in X \times \mathbb{R} \mid Bv=0, \ \langle u^*_j, v \rangle \leq 0, \ j=1,\dots,q, \\
&\hspace*{4.1cm}\langle v^*_k, v \rangle - \mu \leq 0, \ k=1,\dots,m \big\}\\
&=\left\{(v, \mu) \in X \times \mathbb{R} \mid v\in  0^{+}({\rm dom}\,f), \ \langle v^*_k, v \rangle\leq\mu,\ k=1,\dots,m \right\}.
\end{aligned}
\end{equation*}
From this and \eqref{eq_rec_fun} we obtain \eqref{eq_rec_gpcf}. $\hfill\Box$
\end{proof}

\noindent
{\it Proof of Theorem~\ref{E_exist_thm}} First, suppose that ${\rm (}{\mathcal{P}}{\rm )}$ has a solution $x_0$. Let $v \in 0^+D$ be given arbitrarily. If $v \notin 0^{+}({\rm dom}\,f)$, then $f0^{+}(v)=+\infty$ by Lemma~\ref{rep_rec_gpcf}. If $v \in 0^{+}({\rm dom}\,f)$, then $f0^{+}(v)=\max \left\{ \langle v_k^*, v \rangle\mid k=1,\dots,m \right\}$ by Lemma~\ref{rep_rec_gpcf}. Select any $t>0$. Since $x_0+tv \in D \cap {\rm dom}\,f$, one has
\begin{equation*}
\begin{aligned}
f(x_0) &\leq f(x_0+tv)=\max \left\{ \langle v_k^*, x_0 \rangle + \beta_k + t  \langle v_k^*, v \rangle\mid k=1,\dots,m \right\} \\
&\leq \max \left\{ \langle v_k^*, x_0 \rangle + \beta_k \mid k=1,\dots,m \right\} + \max \left\{ t  \langle v_k^*, v \rangle\mid k=1,\dots,m \right\}\\
&=f(x_0)+t f0^{+}(v).
\end{aligned}
\end{equation*}   
It follows that $f0^{+}(v) \geq 0$. 

Conversely, suppose that $f0^{+}(v) \geq 0$ for every $v \in 0^+D$. Since $D \cap {\rm dom}\,f$ is a nonempty generalized polyhedral convex set, by the representation theorem for gpcs \cite[Theorem~2.7]{Luan_Yen_2015}, one can find $u_1, \dots, u_d$ in $D \cap {\rm dom}\,f$, $v_1, \dots, v_{\ell}$ in $X$, and a closed linear subspace $X_0 \subset X$ such that 
\begin{equation}\label{rep-1} 
D \cap {\rm dom}\,f={\rm conv}\,\{u_1, \dots, u_d\} +{\rm cone}\,\{v_1, \dots, v_{\ell}\} + X_0.
\end{equation} 
Then, $0^{+}(D \cap {\rm dom}\,f)={\rm cone}\,\{v_1, \dots, v_{\ell}\} + X_0.$ Put
$$\gamma=\min \left\{ \langle v_k^*, u_i \rangle + \beta_k \mid k=1,\dots,m, \, i=1,\dots,d \right\}.$$
One has $f(x) \geq \gamma $ for every $x \in D$. Indeed, if $x \notin {\rm dom}\,f$, then the inequality is obvious, because $f(x)=+\infty$. Now, suppose that $x \in D \cap {\rm dom}\,f$. According to \eqref{rep-1}, there exist $\lambda_1 \geq 0, \dots, \lambda_d \geq 0$, and $v \in 0^{+}(D \cap {\rm dom}\,f)$ satisfying $\sum\limits_{i=1}^d \lambda_i=1$ and $x=\sum\limits_{i=1}^d \lambda_i u_i + v$. For each $k=1,\dots,m$, one has
\begin{equation*}
\begin{aligned}
\langle v_k^*, x \rangle + \beta_k &= \sum\limits_{i=1}^d \lambda_i \langle v_k^*, u_i \rangle + \langle v_k^*, v \rangle + \beta_k= \sum\limits_{i=1}^d \lambda_i  \big(\langle v_k^*, u_i \rangle + \beta_k\big) +\langle v_k^*, v \rangle\\
& \geq \sum\limits_{i=1}^d \lambda_i \gamma +\langle v_k^*, v \rangle = \gamma +\langle v_k^*, v \rangle.
\end{aligned}
\end{equation*}   
Consequently, $$\max\{ \langle v_k^*, x \rangle + \beta_k  \mid k=1,\dots,m \} \geq \max\{ \gamma +\langle v_k^*, v \rangle \mid k=1,\dots,m \}.$$  
Combing this with \eqref{eq_rep_gcpf}, we obtain 
\begin{equation*}
f(x)\geq \max\{\gamma +\langle v_k^*, v \rangle  \mid k=1,\dots,m \} =\gamma + \max\{\langle v_k^*, v \rangle  \mid k=1,\dots,m \}.
\end{equation*}   
Since $v \in  0^{+}(D \cap {\rm dom}\,f)$, one has $v \in  0^{+}({\rm dom}\,f)$. So, by Lemma~\ref{rep_rec_gpcf}, $$\max \left\{ \langle v_k^*, v \rangle\mid k=1,\dots,m \right\}=f0^{+}(v).$$ Then, for every $x \in D \cap {\rm dom}\,f$ we have $f(x) \geq \gamma + f0^{+}(v) \geq \gamma$, where the last inequality holds because $v \in  0^{+}(D) $. Thus, by Theorem~\ref{FW_exist_thm}, ${\rm (}{\mathcal{P}}{\rm )}$ has a solution. $\hfill\Box$

\begin{remark} If $f0^{+}(v) \geq 0$ for every $v \in 0^+D$, then one says that the functional $f0^{+}$ is {\it copositive} on the recession cone $0^+D$. We call Theorem~\ref{E_exist_thm} an Eaves-type existence theorem in generalized polyhedral convex optimization to trace back Eaves' idea \cite[p.~702]{Eaves_1971} (see also \cite[Theorem~2.2]{Lee_Tam_Yen_2005}) in using recession cones for a solution existence theorem in quadratic programming. In the special case where $f$ is linear on $X$, the result in Theorem~\ref{E_exist_thm} has been obtained in \cite[Theorem~3.1]{Luan_Yen_2015}.
\end{remark}

We now give an explicit criterion for ${\rm (}{\mathcal{P}}{\rm )}$ to have a solution.
\begin{theorem}\label{Exist_thm_3} Let $D$ be given by \eqref{eq_def_gpcs_2}, the function $f$ be defined by \eqref{eq_rep_gcpf} with ${\rm dom}\, f$ be given by \eqref{rep_domf}. Suppose that $D \cap {\rm dom}\,f$ is nonempty. Then ${\rm (}{\mathcal{P}}{\rm )}$ has a solution if and only if 
	\begin{equation}\label{eq_exist_thm_3}\begin{array}{rcl}
	0 \in {\rm conv}\,\{v_k^* \mid  k=1,\dots,m\}&+&{\rm cone}\,\{u^*_j \mid  j=1,\dots,q \}\\
	& + & {\rm cone}\,\{x^*_i  \mid  i=1,\dots,p\}+ ({\rm ker}\,A \cap {\rm ker}\,B)^{\perp}.
	\end{array}
		\end{equation}
\end{theorem}
\begin{proof} First, suppose that \eqref{eq_exist_thm_3} is fulfilled. Then, there exist  nonnegative numbers $\lambda_1,\dots, \lambda_m,$ $\mu_{1,1}, \dots, \mu_{1,p}$, $\mu_{2,1}, \dots, \mu_{2,q}$, and an element $u^* \in ({\rm ker}\,A \cap {\rm ker}\,B)^{\perp}$ such that $\sum\limits_{k=1}^m \lambda_k=1$ and
$$\sum\limits_{k=1}^m \lambda_k v_k^*+\sum\limits_{i=1}^p \mu_{1,i} x_i^* + \sum\limits_{j=1}^q \mu_{2,j} u_j^* + u^*=0.$$ Select any $x_0$ from $D \cap {\rm dom}\,f$. For every $x \in D \cap {\rm dom}\,f$, one has
\begin{equation*}
\begin{aligned}
f(x)&=\max \left\{ \langle v_{\ell}^*, x \rangle + \beta_{\ell} \mid \ell=1,\dots,m  \right\}\\
&=\left(\sum\limits_{k=1}^m \lambda_k\right) \max \left\{ \langle v_{\ell}^*, x \rangle + \beta_{\ell} \mid \ell=1,\dots,m  \right\}\\
&=\sum\limits_{k=1}^m \big(\lambda_k  \max \left\{ \langle v_{\ell}^*, x \rangle + \beta_{\ell} \mid \ell=1,\dots,m  \right\}\big)\\
&\geq \sum\limits_{k=1}^m \lambda_k [\langle v_k^*, x \rangle + \beta_k]= \left\langle \sum\limits_{k=1}^m \lambda_k  v_k^*, x \right\rangle + \sum\limits_{k=1}^m \lambda_k \beta_k\\
&= \left\langle -\left(\sum\limits_{i=1}^p \mu_{1,i} x_i^* + \sum\limits_{j=1}^q \mu_{2,j} u_j^* + u^*\right), x \right\rangle + \sum\limits_{k=1}^m \lambda_k \beta_k\\
&=-\sum\limits_{i=1}^p \mu_{1,i} \langle  x_i^*, x  \rangle - \sum\limits_{j=1}^q \mu_{2,j} \langle  u_j^*, x  \rangle - \langle  u^*, x_0  \rangle +\langle  u^*, x_0 -x \rangle + \sum\limits_{k=1}^m \lambda_k \beta_k\\
&\geq -\sum\limits_{i=1}^p \mu_{1,i} \alpha_i - \sum\limits_{j=1}^q \mu_{2,j} \gamma_j - \langle  u^*, x_0  \rangle +0 + \sum\limits_{k=1}^m \lambda_k \beta_k.
\end{aligned}
\end{equation*}
Hence, $f$ is bounded from below on $D$. Invoking Theorem~\ref{FW_exist_thm}, we conclude that ${\rm (}{\mathcal{P}}{\rm )}$ has a solution. Thus, \eqref{eq_exist_thm_3} implies the solution existence of ${\rm (}{\mathcal{P}}{\rm )}$.

To complete the proof, it suffices to show that if  \eqref{eq_exist_thm_3} does not hold, then ${\rm (}{\mathcal{P}}{\rm )}$ has no solutions. Suppose that $0 \notin Q$, where $Q$ denotes the set on the right-hand side of \eqref{eq_exist_thm_3}. By  \cite[Theorem~2.7]{Luan_Yen_2015}, the nonempty set $Q$ is generalized polyhedral convex. Hence, $Q$ is convex and weakly$^*$-closed. Since $\{0\} \cap Q=~\emptyset$, by the strong separation theorem \cite[Theorem~3.4(b)]{Rudin_1991} one can find $v \in X$ and $\gamma \in \mathbb{R}$ such that 
\begin{equation}\label{separation_Q}
\sup\{\langle x^*,v \rangle \mid  x^* \in Q\} < \gamma <\langle 0,v \rangle.
\end{equation}
\hskip 0.5cm On one hand, \eqref{separation_Q} assures that the linear functional $\langle \cdot,v\rangle $ is bounded from above on~$Q$. Hence, according to \cite[Theorem~3.3]{Luan_Yen_2015}, the generalized linear programming problem
$\max\{\langle x^*,v \rangle \mid  x^* \in Q\}$ has a solution. Therefore, by \cite[Proposition~3.5]{Luan_Yen_2015}, one has  $\langle v^*, v \rangle \leq 0$ for every vector $v^*$ from the recession cone $0^+Q$ of $Q$. As \eqref{eq_exist_thm_3} yields
$$0^+Q={\rm cone}\,\{u^*_j \mid  j=1,\dots,q \} + {\rm cone}\,\{x^*_i  \mid  i=1,\dots,p\}+ ({\rm ker}\,A \cap {\rm ker}\,B)^{\perp},$$  
one gets $\langle x^*_i, v \rangle \leq 0$ for all $i\in \{1,\dots,p\}$, $\langle u^*_j, v \rangle \leq 0$ for every $j \in \{1,\dots,q\}$, and $v \in (({\rm ker}\,A \cap {\rm ker}\,B)^{\perp})^{\perp}$.
Since the linear subspace ${\rm ker}\,A \cap  {\rm ker}\,B$ is closed, by using \cite[Proposition~2.40]{Bonnans_Shapiro_2000} one has $$(({\rm ker}\,A \cap {\rm ker}\,B)^{\perp})^{\perp}={\rm ker}\,A \cap  {\rm ker}\,B.$$  Hence, applying  Remark~\ref{rec_cone_gpcs} simultaneously to $D$ and ${\rm dom}\,f$, we obtain $v \in 0^{+}D$ and $v \in 0^{+}({\rm dom}\,f)$. So, by the second inclusion and by Lemma~\ref{rep_rec_gpcf},  
\begin{equation}\label{f0_max}
f0^{+}(v)=\max \left\{ \langle v_k^*, v \rangle \mid k=1,\dots,m \right\}.
\end{equation}
\hskip 0.5cm On the other hand, for each $k=1,\dots,m$, since $v_k^* \in Q$,  the inequalities in \eqref{separation_Q} yield $\langle v_k^*, v \rangle < \gamma <0$. So, from \eqref{f0_max} it follows that $ f0^{+}(v) < 0$. Hence ${\rm Sol}{\rm (}{\mathcal{P}}{\rm )} = \emptyset$ by Theorem~\ref{E_exist_thm}. 

The proof is complete. $\hfill\Box$
\end{proof}

\begin{corollary}\label{cor_existence} In the notations of Theorem \ref{Exist_thm_3}, suppose that ${\rm dom}\,f \subset D$. Then, ${\rm (}{\mathcal{P}}{\rm )}$ has a solution if and only if 
	\begin{equation*}
	\begin{aligned}
	0 \in {\rm conv}\,\{v_k^* \mid  k=1,\dots,m\}+{\rm cone}\,\{u^*_j \mid  j=1,\dots,q \}+ ({\rm ker}\,B)^{\perp}.
	\end{aligned}
	\end{equation*}
\end{corollary}
\begin{proof} As ${\rm dom}\,f \subset D$,  ${\rm (}{\mathcal{P}}{\rm )}$ is equivalent to the problem $\min \left\{ f(x) \mid x \in {\rm dom}\, f\right\}$. Hence, applying Theorem \ref{Exist_thm_3} the latter, we obtain the assertion. $\hfill\Box$	
\end{proof}	

\begin{corollary}\label{solution_existence_unconstrained} Suppose that $D=X$ and $f$ is given by \eqref{eq_rep_gcpf} with ${\rm dom}\,f=X$. Then ${\rm (}{\mathcal{P}}{\rm )}$ has a solution if and only if $0 \in {\rm conv}\, \{v_k^* \mid k=1,\dots,m\}.$ 
\end{corollary}
\begin{proof} Since ${\rm dom}\,f=X$, we can choose $Z=\{0\}$, $B\equiv 0$, $z=0$ and $q=0$. Therefore, by using Corollary \ref{cor_existence} we obtain the assertion.$\hfill\Box$	
\end{proof}	

Next, we will describe the solution set of ${\rm (}{\mathcal{P}}{\rm )}$.
\begin{proposition} ${\rm Sol}{\rm (}{\mathcal{P}}{\rm )}$ is a generalized polyhedral convex set. If $D$ and ${\rm dom}\, f$ are polyhedral convex, so is ${\rm Sol}{\rm (}{\mathcal{P}}{\rm )}$.
\end{proposition}
\begin{proof} If ${\rm Sol}{\rm (}{\mathcal{P}}{\rm )}$ is empty, then the claim is trivial. If ${\rm Sol}{\rm (}{\mathcal{P}}{\rm )}$ is nonempty, select a point $\bar{x} \in {\rm Sol}{\rm (}{\mathcal{P}}{\rm )}$ and put $ \bar{\gamma}=f(\bar{x})$. Then, $f(x) \geq \bar\gamma$ for every $x \in D$. With $f$ being given by \eqref{eq_rep_gcpf}, one has
\begin{equation}\label{solution_set}
\begin{aligned}
{\rm Sol}{\rm (}{\mathcal{P}}{\rm )}&=\{x \in D \mid  f(x) =\bar{\gamma}\}=\{x \in D \mid  f(x) \leq \bar{\gamma}\}\\
&=\{x \in D \cap {\rm dom}\, f \mid  f(x) \leq \bar{\gamma}\}\\
&=\{x \in D \cap {\rm dom}\, f \mid \langle v_k^*, x \rangle + \beta_k \leq \bar{\gamma},\; k=1,\dots,m\}.
\end{aligned}
\end{equation}
Since ${\rm dom}\,f $ is a generalized polyhedral convex set (see Theorem~\ref{properties_gpcf}), this implies that ${\rm Sol}{\rm (}{\mathcal{P}}{\rm )}$ is a generalized polyhedral convex set. In the case where $D$ and ${\rm dom}\,f$ are polyhedral convex, \eqref{solution_set} shows that ${\rm Sol}{\rm (}{\mathcal{P}}{\rm )}$ is a pcs. $\hfill\Box$
\end{proof}

The following example is an illustration for our results in this section.

\begin{example}{\rm (Cf. \cite[Example 1]{Luan_2016_Acta} and \cite[Example 2.8]{Luan_Yen_2015})} \label{Example1_gpcp}
Let $X=C[-1, 1]$ be the Banach space of continuous real-valued functions defined on $[-1, 1]$ with the norm $||x||=\max\limits_{t \in [-1, 1]} |x(t)|.$ By the Riesz representation theorem (see, e.g., \cite[Theorem~6,~p.~374]{Kolmogorov_Fomin_1975} and \cite[Theorem~1, p.~113]{Luenberger_1969}), the dual space of $X$ is $X^*=NBV[-1, 1]$, the space of \textit{functions of bounded variation} on $[-1, 1]$, i.e., functions $y: [-1, 1] \rightarrow \mathbb{R}$ of bounded variation, $y(-1)=0$, and $y(\cdot)$ is continuous from the left at every point of $(-1, 1)$. To construct a generalized polyhedral convex optimization problem on $X$, we first define the elements $x^*_1,\, x^*_2 \in X^*$ by setting 
\begin{equation}
\label{integrals_1_2}\langle x^*_i, x \rangle =\int\limits_{-1}^1 t^i x(t)dt\quad (i=1,2),
\end{equation} 
where the integrals are Riemannian. For each index $i\in\{1,2\}$, the corresponding integral in \eqref{integrals_1_2} is equal to the Riemann-Stieltjes integral $\int\limits_{-1}^1 x(t)dy_i(t)$, which is given by the $C^1$-smooth functions $y_i(t)=\int\limits_{-1}^t \tau^i d \tau$ (see, e.g., \cite[p.~367]{Kolmogorov_Fomin_1975}). Consider~$X$ with the \textit{weak topology}. Then $X$ is a locally convex Hausdorff topological vector space whose topology is much weaker than the norm topology. Clearly, $X_0:={\rm ker}\,x^*_1 \cap {\rm ker}\,x^*_2$ is a closed linear subspace of~$X$. Let 
\begin{equation*}
		e_1(t)=\begin{cases}
		0 & \text{if} \ \, t \in [-1,0]\\
		-60t^2+48t   & \text{if} \ \, t \in [0,1], \\ 
		\end{cases}	
\end{equation*} 
and
\begin{equation*} 
		e_2(t)=\begin{cases}
		0 & \text{if} \ \, t \in [-1,0]\\
		80t^2-60t   & \text{if} \ \, t \in [0,1]. \\ 
		\end{cases}
\end{equation*}
We have $e_1, e_2\in X$, $\langle x^*_1,e_1 \rangle= \langle x^*_2,e_2 \rangle=1,$ and $\langle x^*_1,e_2 \rangle= \langle x^*_2,e_1 \rangle=0$. For any $x \in X$, put $t_i=\langle x^*_i,x \rangle$ for $i=1,2$, and observe that the vector $x_0:=x-t_1e_1-t_2e_2$ belongs to~$X_0$. Conversely, if $x=x_0+t_1e_1+t_2e_2$, with $x_0 \in X_0$ and $t_1, t_2 \in \mathbb{R}$, then
$$\langle x^*_i, x \rangle = \langle x^*_i, x_0 \rangle + t_1\langle x^*_i, e_1 \rangle+ t_2\langle x^*_i, e_2 \rangle=t_i \quad(i=1,2).$$ Therefore, for any $x \in X$, there exists a unique tube $(x_0, t_1, t_2) \in X_0 \times \mathbb{R} \times \mathbb{R}$ such that $x=x_0+t_1e_1+t_2e_2$. Given any $e_0 \in X_0$ and put $L=\left\{x \in X \mid x(t)=e_0(t),\, t \in [-1,0] \right\}$. Clearly, $L$ is a closed affine subspace of $X$. Consider ${\rm (}{\mathcal{P}}{\rm )}$ with the constraint set $$D:=\left\{x \in L \mid \langle x^*_1, x \rangle \leq 1, \; \langle x^*_2, x \rangle \leq 2\right\}.$$ Observe that $D$ is a gpcs, $e_0 + e_2 \in D$, and $D$ is not a polyhedral convex set in $X$. To define the objective function, choose $v^*_1=x^*_1-x^*_2$, $v^*_2=-x^*_1-x^*_2$, and put 
		\begin{equation*}
		f(x)=\max\{\langle v^*_1, x \rangle+1, \langle v^*_2, x \rangle\} \quad (x \in X).
		\end{equation*}
For any $x \in D$, we have
		\begin{equation*}
		\begin{aligned}
		f(x) &\geq \frac{1}{2}[\langle v^*_1, x \rangle+1] + \frac{1}{2}\langle v^*_2, x \rangle\\
		&=\frac{1}{2} \langle v^*_1+v^*_2, x \rangle +  \frac{1}{2}=\langle -x^*_2, x \rangle +  \frac{1}{2}\geq -\frac{3}{2}.
		\end{aligned}
		\end{equation*}
Thus, by Theorem \ref{FW_exist_thm}, ${\rm (}{\mathcal{P}}{\rm )}$ has a solution. 
\end{example}

\section{Optimality Conditions}
\setcounter{equation}{0}  

We now obtain some optimality conditions for ${\rm (}{\mathcal{P}}{\rm )}$.
\begin{theorem}\label{Optimality_condition_1} {\rm (Optimality condition I)} A vector $x \in D \cap {\rm dom}\,f$ is a solution of ${\rm (}{\mathcal{P}}{\rm )}$ if and only if 
	\begin{equation}\label{eq_optim_cond}
	0 \in \overline{\partial f(x) + N_D(x)}.
	\end{equation}
\end{theorem}
\begin{proof} Clearly, ${\rm Sol}{\rm (}{\mathcal{P}}{\rm )}$ coincides with the solution set of a problem 
$${\rm (}{\mathcal{P'}}{\rm )} \quad \min\{f(x)+\delta(x, D) \mid x \in X\}.$$ 
Since the functions $f$, $\delta(\cdot, D)$ are proper generalized polyhedral convex and since $D \cap {\rm dom}\,f\neq \emptyset$, the function $\widetilde{f}:=f+\delta(\cdot, D)$ is proper generalized polyhedral convex by \cite[Theorem~3.7]{Luan_Yao_Yen_2016}.  On one hand, by Lemma~\ref{sum_rules} we have $$\partial \widetilde{f}(x)=\overline{\partial f(x) + N_D(x)} \quad (x \in D \cap {\rm dom}\,f).$$ On the other hand, since $\widetilde{f}$ is proper convex, a vector $x\in X$ belongs to ${\rm Sol}{\rm (}{\mathcal{P'}}{\rm )}$ if and only if $0 \in \partial \widetilde{f}(x)$; see \cite[Proposition 1, p. 81]{Ioffe_Tihomirov_1979}. Therefore, vector $x \in D \cap {\rm dom}\,f$ is a solution of ${\rm (}{\mathcal{P}}{\rm )}$ if and only if $0 \in \partial \widetilde{f}(x)=\overline{\partial f(x) + N_D(x)}.$

The proof is complete.  $\hfill\Box$
\end{proof}

One may ask: \textit{The closure sign in \eqref{eq_optim_cond} can be omitted, or not?} If $X$ is finite-dimensional, then $f$ and $\delta(\cdot, D)$ are polyhedral convex functions. So, by Lemma~\ref{sum_rules}, $$\partial (f+\delta(\cdot, D))(x)=\partial f(x) + N_D(x)\quad (x \in D \cap {\rm dom}\,f).$$ So, the closure sign in~\eqref{eq_optim_cond} is superfluous. However, as shown in next example, if $X$ is infinite-dimension\-al, then the closure sign in~\eqref{eq_optim_cond} is essential. 

\begin{example}\label{Ex_optim_cond} According to \cite[Example~3.34]{Bauschke_Combettes_2011}, one can find an infinite-dimensional Hilbert space $X$ and two suitable closed linear subspaces $X_1, X_2$ of $X$ with  $\overline{X_1+X_2}=X$ and $X_1+X_2 \neq X$. Let $D_i$ be the \textit{orthogonal complement} of $X_i$, i.e.,  $$D_i=\{x\in X\mid \langle x,u\rangle=0,\ \forall u\in X_i\}, \quad (i=1,2).$$ It is clear that $D_1, D_2$ are generalized polyhedral convex sets in $X$ and $D_1 \cap D_2=\{0\}$. Since $X_1+X_2 \neq X$, there exists $v^* \in X\setminus(X_1+X_2)$. Hence, $-v^* \notin X_1+X_2$ because $X_1+X_2$ is a linear subspace. Consider a generalized polyhedral convex optimization problem ${\rm (}{\mathcal{P}}{\rm )}$ with 
		\begin{equation*}
		f(x)=\begin{cases}
		\langle v^*, x \rangle & \text{ if } x \in D_1\\
		+\infty & \text{ if } x \notin D_1,
		\end{cases}
		\end{equation*}	
and $D=D_2$. Obviously, ${\rm Sol}{\rm (}{\mathcal{P}}{\rm )}=\{0\}$. Note that $$\partial f(0)=v^*+N_{D_1}(0)=v^*+X_1.$$ Combining this with the equality $N_{D}(0)=X_2$, one has $\partial f(0)+N_{D}(0)=v^*+X_1+X_2$. The inclusion $-v^* \notin X_1+X_2$ yields $0 \notin v^*+X_1+X_2$. It follows that $x=0$ is a solution of ${\rm (}{\mathcal{P}}{\rm )}$, but $0 \notin \partial f(0)+N_{D}(0)$.
\end{example}

Note that, if $f$ is a polyhedral convex function or $D$ is a polyhedral convex set, then $\partial (f+\delta(\cdot, D))(x)=\partial f(x) + N_D(x)$ for all $x \in D \cap {\rm dom}\,f$ by Lemma~\ref{sum_rules}. Thus, the following statement holds.
\begin{theorem}\label{Optimality_condition_2_pc} {\rm (Optimality condition II)}  Assume that either $f$ is a proper polyhedral convex function or $D$ is  polyhedral convex set. Then, $x \in D \cap {\rm dom}\,f$ is a solution of~${\rm (}{\mathcal{P}}{\rm )}$ if and only if $0 \in \partial f(x) + N_D(x).$
\end{theorem}

The forthcoming example is designed as an illustration for Theorem~\ref{Optimality_condition_2_pc}.

\begin{example}\label{solve_Example1_gpcp} Let ${\rm (}{\mathcal{P}}{\rm )}$ be the problem described in Example~\ref{Example1_gpcp}. To solve it, we first compute the set $\partial f(x) + N_D(x)$ for every $x \in D$. Clearly, $f$ is a polyhedral convex function with ${\rm dom}\,f=X$. Hence, ${\rm dom}\,f$ can be given by \eqref{rep_domf}, where $Z=\{0\},B\equiv 0, z=0$ and $q=0$. Therefore, by~\eqref{rep_subd} one gets $$\partial f(x)={\rm conv}\left\{v_k^* \mid  k \in \Theta(x)\right\} \quad (x \in X).$$ Since~$L$ is a closed affine subspace of~$X$, by \cite[Remark~2.196]{Bonnans_Shapiro_2000}, one can find a continuous surjective linear mapping~$A$ from $X$ to a lcHtvs $Y$ and a vector $y \in Y$ satisfying $L=\left\{x \in X \mid Ax=y \right\}$. It is easy to verify that ${\rm ker}\,A=L-e_0$. Combining this with \eqref{norm_cone_gpcs}, we obtain $N_D(x)={\rm cone}\left\{x_i^* \mid  i \in I(x)\right\} +(L-e_0)^{\perp}$ for every $x \in D$. Hence,
		\begin{equation}\label{Solve_Ex1_sum_sub_normalcone}
		\partial f(x) + N_D(x)={\rm conv}\left\{v_k^* \mid  k \in \Theta(x)\right\}+{\rm cone}\left\{x_i^* \mid  i \in I(x)\right\} +(L-e_0)^{\perp}
		\end{equation}		 
for every $x \in D$. Now, suppose that $x$ is a solution of ${\rm (}{\mathcal{P}}{\rm )}$.  The ``only if'' part of Theorem~\ref{Optimality_condition_2_pc} tells us that $0 \in \partial f(x) + N_D(x)$. Then, due to  \eqref{Solve_Ex1_sum_sub_normalcone}, we have $$0=\lambda_1 v^*_1+\lambda_2v^*_2+\mu_1x^*_1 +\mu_2 x^*_2 +x^*_0,$$ 
where $\lambda_1 \geq 0, \lambda_2 \geq 0$, $\lambda_1+\lambda_2=1$, $\mu_1 \geq 0$, $\mu_2 \geq 0$, $x^*_0 \in (L-e_0)^{\perp}$, with $\lambda_k=0$ if $k \notin \Theta(x)$, and $\mu_i=0$ if $i \notin I(x)$. Since $v^*_1=x^*_1-x^*_2$ and $v^*_2=-x^*_1-x^*_2$, one has
		\begin{equation*}\label{Example2_rep}
		(\lambda_1-\lambda_2+\mu_1)x^*_1+(\mu_2-1)x^*_2+x^*_0=0.
		\end{equation*}
Therefore, $(\lambda_1-\lambda_2+\mu_1) \langle x^*_1, e_i \rangle + (\mu_2-1)\langle x^*_2, e_i \rangle +\langle x^*_0, e_i \rangle =0$ with $i=1,2$. For each index $i \in \{1,2\}$, since $e_0+e_i \in L$, we must have $\langle x^*_0, e_i \rangle =0$. Consequently, $\lambda_1-\lambda_2+\mu_1=0$ and $\mu_2-1=0$. The latter fact yields $2 \in I(x)$, i.e., $\langle x^*_2, x \rangle=2$. We observe that $1 \notin I(x)$. Indeed, on the contrary, suppose that $1 \in I(x)$, i.e., $\langle x^*_1, x \rangle=1$. Then, $\langle v^*_1, x \rangle+1=0$ and $\langle v^*_2, x \rangle =-3$; so $f(x)=0$ and $\Theta(x)=\{1\}$. Thus $\lambda_2=0$,  $\lambda_1=1$ and $\mu_1=-1 <0$, a contradiction. Since $1 \notin I(x)$, we must have $\mu_1=0$ and $\lambda_1=\lambda_2=\frac{1}{2}$; hence $\Theta(x)=\{1,2\}$. This means that $f(x)=\langle v^*_1, x \rangle +1 =  \langle v^*_2, x \rangle$. Of course, $\langle x^*_1, x \rangle=-\frac{1}{2}$. We have thus proved that if $x \in {\rm Sol(}{\mathcal{P}}{\rm )}$, then $\langle x^*_1, x \rangle=-\frac{1}{2}$ and $\langle x^*_2, x \rangle=2$. Conversely, if $u \in D$ satisfies $\langle x^*_1, u \rangle=-\frac{1}{2}$ and $\langle x^*_2, u \rangle=2$, then one has $\Theta(u)=\{1,2\}$ and $I(u)=\{2\}$. By \eqref{Solve_Ex1_sum_sub_normalcone}, $$0=\frac{1}{2}v^*_1+\frac{1}{2}v^*_2 + x^*_2+0 \in \partial f(u) + N_D(u).$$ 
Therefore, the ``if'' part in Theorem \ref{Optimality_condition_2_pc} shows that  $u$ is a solution of ${\rm (}{\mathcal{P}}{\rm )}$.  Thus,
\begin{equation*}
{\rm Sol}{\rm (}{\mathcal{P}}{\rm )}=\left\{u \in L \mid \langle x^*_1, u \rangle = -\frac{1}{2}, \,  \langle x^*_2, u \rangle =2\right\}.
\end{equation*}
Using this formula, one can verify that $e_0-\frac{1}{2}e_1+2e_2 \in {\rm Sol}{\rm (}{\mathcal{P}}{\rm )}$. Thus, thanks to Theorem~\ref{Optimality_condition_2_pc}, we have found the formula for the solution set of ${\rm (}{\mathcal{P}}{\rm )}$  and showed that it is nonempty. The optimal value of ${\rm (}{\mathcal{P}}{\rm )}$ is $-\frac{3}{2}$. 
\end{example}

Under the assumptions of Theorem~\ref{Optimality_condition_2_pc}, by \cite[Proposition~2.11]{Luan_Yao_Yen_2016} we know that $D-{\rm dom}\,f$ is a pcs in~$X$. We want to have an analogue of Theorem~\ref{Optimality_condition_2_pc} in a Banach space setting for the case $D-{\rm dom}\,f$ is a gpcs. Next lemma is useful for the proof of the desired result.
\begin{lemma}\label{sum_affinesub} Let $L_1, L_2$ be closed affine subspaces, $P_1, P_2$ polyhedral convex sets in~$X$. Suppose that $D_1:=L_1 \cap P_1$ and $D_2:=L_2 \cap P_2$ are nonempty. If $D_1-D_2$ is a generalized polyhedral convex set in $X$, so is $L_1-L_2$.    	
\end{lemma}
\begin{proof} For any $i \in \{1,2\}$ and $x_i \in D_i$, observe that $D_i':=D_i-x_i$ is a nonempty pcs in the closed linear subspace $M_i:=L_i-x_i$. By \cite[Lemma~2.12]{Luan_Yao_Yen_2016}, there exist $u_{i,1},\dots,u_{i,k_i}$ in $M_i$,  $v_{i,1},\dots,v_{i,\ell_i}$ in $M_i$, and a closed linear subspace $M_{i,0}$ of finite codimension of~$M_i$ such that 
$$D_i'={\rm conv}\,\left\{u_{i,1}, \dots,u_{i,k_i}\right\}+{\rm cone}\, \left\{v_{i,1}, \dots, v_{i,\ell_i}\right\}+M_{i,0}.$$  
Due to the finite codimension property of $M_{i,0}$ in $M_i$, one can find  $x_{i,1}, \dots, x_{i,m_i}$ in $M_i$ such that $M_i= M_{i,0} + {\rm span} \left\{x_{i,1}, \dots, x_{i,m_i}\right\}.$
Hence,
\begin{equation*}
\begin{aligned}
M_i &\supset D_i'+{\rm span} \left\{x_{i,1}, \dots, x_{i,m_i}\right\}\\
&={\rm conv}\,\left\{u_{i,1}, \dots,u_{i,k_i}\right\}+{\rm cone}\, \left\{v_{i,1}, \dots, v_{i,\ell_i}\right\}+M_{i,0}+{\rm span} \left\{x_{i,1}, \dots, x_{i,m_i}\right\}\\
&={\rm conv}\,\left\{u_{i,1}, \dots,u_{i,k_i}\right\}+{\rm cone}\, \left\{v_{i,1}, \dots, v_{i,\ell_i}\right\}+M_i\\
&=M_i.
\end{aligned}
\end{equation*}
This forces $M_i=D_i'+{\rm span} \left\{x_{i,1}, \dots, x_{i,m_i}\right\}$. Consequently, 
\begin{equation*}
\begin{aligned}
L_i&=x_i+M_i=x_i+D_i'+{\rm span} \left\{x_{i,1}, \dots, x_{i,m_i}\right\}\\
&=D_i+{\rm span} \left\{x_{i,1}, \dots, x_{i,m_i}\right\}.
\end{aligned}
\end{equation*}
It is clear that $-L_2=(-D_2)+{\rm span} \left\{x_{2,1}, \dots, x_{2,m_2}\right\}.$ Therefore,
\begin{equation*}
L_1-L_2=(D_1-D_2)+{\rm span} \left\{x_{1,1}, \dots, x_{1,m_1}, x_{2,1}, \dots, x_{2,m_2} \right\}.
\end{equation*}
Since $D_1-D_2$ is a gpcs by our assumption  and ${\rm span} \left\{x_{1,1}, \dots, x_{1,m_1}, x_{2,1}, \dots, x_{2,m_2} \right\}$ is a finite-dimensional subspace, $L_1-L_2$ is a generalized polyhedral convex set; see \cite[Proposition~2.10]{Luan_Yao_Yen_2016}. $\hfill\Box$
\end{proof}

\begin{theorem}\label{Optimality_condition_3_Banach} {\rm (Optimality condition III)}  Suppose that $X$ is a Banach space and the set $D-{\rm dom}\,f$ is generalized polyhedral convex. Then, $x \in D \cap {\rm dom}\,f$ is a solution of ${\rm (}{\mathcal{P}}{\rm )}$ if and only if $0 \in \partial f(x) + N_D(x).$
\end{theorem}
\begin{proof} Let ${\rm dom}\,f$  be described by \eqref{rep_domf}. Put
$$L_1=\{x \in X \mid Ax=y\},\ L_2=\{x \in X \mid Bx=z\}$$
and
\begin{equation*}
\begin{aligned}
P_1=\{x \in X \mid  \langle x^*_i, x \rangle \leq \alpha_i,\  i=1,\dots,p\},\ P_2=\{x \in X \mid  \langle u^*_j, x \rangle \leq \gamma_j,\  j=1,\dots,q\}. 
\end{aligned}
\end{equation*} 
Clearly, $L_1,\, L_2$ are closed affine subspaces, and $P_1,\, P_2$ are  polyhedral convex sets. One has $D=L_1 \cap P_1$ and ${\rm dom}\,f=L_2 \cap P_2$. Since $D-{\rm dom}\,f$ is a generalized polyhedral convex set, $L_1-L_2$ is a gpcs by Lemma~\ref{sum_affinesub}. For every $i \in \{1,2\}$, select a point $x_i \in L_i$. Obviously, 
\begin{equation*}
\begin{aligned}
{\rm ker}\, A + {\rm ker}\, B&={\rm ker}\, A - {\rm ker}\, B\\
&=(L_1-x_1)-(L_2-x_2)=(L_1-L_2)-(x_1-x_2).
\end{aligned}
\end{equation*} 
Since $L_1-L_2$ is a gpcs, ${\rm ker}\, A + {\rm ker}\, B$ is a gpcs by \cite[Proposition~2.10]{Luan_Yao_Yen_2016}. In particular, ${\rm ker}\, A + {\rm ker}\, B$ is closed. Hence, by \cite[Theorem~2.16]{Brezis_2011},
\begin{equation}\label{sum_subspaces}
\left( {\rm ker}\,A \right)^{\perp} + \left({\rm ker}\,B\right)^{\perp}=\left( {\rm ker}\,A \cap {\rm ker}\,B\right)^{\perp}.
\end{equation}
Therefore, for every $x \in D \cap {\rm dom}\, f$, from \eqref{norm_cone_gpcs}, \eqref{rep_subd}, and \eqref{sum_subspaces} we obtain
\begin{equation*}\label{sum_subd_normalcone}
\begin{aligned}
\partial f(x) + N_D(x) & = {\rm conv}\left\{v_k^* \mid  k \in \Theta(x)\right\}+ {\rm cone}\,\{u^*_j \mid  j \in J(x)\} + ({\rm ker}\,B)^{\perp}\\
&\hspace*{3.3cm}+{\rm cone}\,\{x^*_i\mid   i \in I(x)\} + ({\rm ker}\,A)^{\perp}\\
& = {\rm conv}\left\{v_k^* \mid  k \in \Theta(x)\right\}+{\rm cone}\,\{x^*_i\mid   i \in I(x)\}\\
&\hspace*{1.5cm}+{\rm cone}\,\{u^*_j \mid  j \in J(x)\}+ \left( {\rm ker}\,A \cap {\rm ker}\,B\right)^{\perp}.\\
\end{aligned}
\end{equation*}
Then, by the representation theorem for gpcs \cite[Theorem~2.7]{Luan_Yen_2015}, we conclude that $\partial f(x) + N_D(x)$ is a generalized polyhedral convex set. So, the latter is closed. Combining this with Theorem~\ref{Optimality_condition_1}, we obtain the assertion.  $\hfill\Box$
\end{proof}

Turning back to the optimality condition given by Theorem~\ref{Optimality_condition_1}, we observe that sometimes it is difficult to find the topological closure of the sum $\partial f(x) + N_D(x)$. The forthcoming theorem gives a new optimality condition for ${\rm (}{\mathcal{P}}{\rm )}$ in the general case, where no topological closure sign is needed.

\begin{theorem}\label{optim_cond_4} {\rm (Optimality condition IV)} A vector $x \in D \cap {\rm dom}\,f$ is a solution of ${\rm (}{\mathcal{P}}{\rm )}$ if and only if
	\begin{equation}\label{eq_optimal_cond_2}
	\begin{array}{rccl}
	0 & \in \, {\rm conv}\left\{v_k^* \mid  k \in \Theta(x)\right\} &+& {\rm cone}\,\{x^*_i\mid   i \in I(x)\}\\
	&&+&{\rm cone}\,\{u^*_j \mid  j \in J(x)\}+ ({\rm ker}\,A \cap {\rm ker}\,B)^{\perp}.
	\end{array}
	\end{equation}	 
\end{theorem}

For proving this result, we will use the directional differentiability property of convex functions and a lemma.

\medskip
It is well known \cite[Proposition~3, p.~194]{Ioffe_Tihomirov_1979} (see also \cite[Theorem~2.1.13]{Zalinescu_2002}) that  the \textit{directional derivative}  $f'(x; h):=\lim\limits_{t\to 0^+} \frac{f(x+th)-f(x)}{t}$ of a proper convex function $f: X \to \bar{\mathbb{R}}$ at $x \in {\rm dom}\,f$ w.r.t. a direction $h \in X$ exists, and one has 
\begin{equation}\label{directional derivative}
f'(x; h)=\inf\limits_{t >0} \frac{f(x+th)-f(x)}{t}.
\end{equation}
(The situation $f'(x; h)=-\infty$ may occur. To have a simple example, one can choose $f(x)=-\sqrt{2x-x^2}$ for $x\in [0,2]$,  $f(x)=+\infty$ otherwise, and note that $f'(0;1)=-\infty$.) According \cite[Theorem~3.9]{Luan_Yao_Yen_2016}, if $f$ is a proper gpcf (resp., a proper pcf), so is $f'(x; \cdot)$. In addition, from \eqref{directional derivative} it follows that $\bar x \in D$ is a solution of ${\rm (}{\mathcal{P}}{\rm )}$ if and only if $f'(\bar x;h) \geq 0$ for every $h \in T_D(\bar x)$. 

\begin{lemma}\label{dir_der_gpcf} If $x \in {\rm dom}\,f$, then
	\begin{equation}\label{eq_dir_der_gpcf}
	f'(x;h) = \begin{cases}
	\max\{\langle v_k^*, h \rangle \mid  k \in \Theta(x)\} & \textrm{ if } \, h \in T_{{\rm dom}\,f}(x)\\
	+\infty & \textrm{ if } \,  h \notin T_{{\rm dom}\,f}(x).
	\end{cases}
	\end{equation}	 
\end{lemma}
\begin{proof} Since ${\rm dom}\,f$ is a nonempty generalized polyhedral convex set, by \cite[Proposition~2.19]{Luan_Yao_Yen_2016} one has $T_{{\rm dom}\,f}(x)={\rm cone}\,[({\rm dom}\,f) -x]$. 

If $h\notin T_{{\rm dom}\,f}(x)$, then $x+th \notin {\rm dom}\,f$ for every $t>0$. So, $f'(x;h)=+\infty$. 

If $h \in T_{{\rm dom}\,f}(x)$, then there exists $\delta_0 >0$ such that $x+th \in {\rm dom}\,f$ for all $t \in [0, \delta_0]$. Therefore, for every $t \in [0, \delta_0]$, by \eqref{eq_rep_gcpf} one has
\begin{equation}\label{dir_der_1}
f(x+th)=\max\left\{\langle v_k^*, x \rangle + \beta_k +t \langle v_k^*, h \rangle \mid  k=1,\dots,m\right\}.
\end{equation}
Select an index $k_0 \in \Theta(x)$. For any $\ell \notin \Theta(x)$, as $f(x)=\langle v_{k_0}^*, x \rangle + \beta_{k_0} >\langle v_{\ell}^*, x \rangle + \beta_{\ell}$, there must exists $\delta_{\ell} > 0$ satisfying
\begin{equation}\label{dir_der_2}
\langle v_{k_0}^*, x \rangle + \beta_{k_0} +t \langle v_{k_0}^*, h \rangle > \langle v_{\ell}^*, x \rangle + \beta_{\ell} +t \langle v_{\ell}^*, h \rangle \quad (t \in [0, \delta_{\ell}]).
\end{equation}
Choose $\delta>0$ such that $\delta \leq \delta_0$ and $\delta \leq \delta_{\ell}$ for all $\ell \notin \Theta(x)$. Then, for every $t \in [0, \delta]$, from \eqref{dir_der_1} and \eqref{dir_der_2} it follows that
\begin{equation}\label{dir_der_3}
\begin{aligned}
f(x+th)&\geq \langle v_{k_0}^*, x \rangle + \beta_{k_0} +t \langle v_{k_0}^*, h \rangle \\
&> \max\left\{\langle v_{\ell}^*, x \rangle + \beta_{\ell} +t \langle v_{\ell}^*, h \rangle \mid  \ell \notin \Theta(x)\right\}.
\end{aligned}
\end{equation}
Thus, combining \eqref{dir_der_1} with \eqref{dir_der_3}, we have
\begin{equation*}
\begin{aligned}
f(x+th)&=\max\left\{\langle v_k^*, x \rangle + \beta_k +t \langle v_k^*, h \rangle \mid  k \in \Theta(x)\right\}\\
&=\max\left\{f(x) +t \langle v_k^*, h \rangle \mid  k \in \Theta(x)\right\}\\
&=f(x)+
t\max\left\{\langle v_k^*, h \rangle \mid  k \in \Theta(x)\right\}
\end{aligned}
\end{equation*}
for all $t \in [0, \delta]$. It follows that $f'(x;h)=\max\left\{\langle v_k^*, h \rangle \mid  k \in \Theta(x)\right\}$. We have thus proved formula \eqref{eq_dir_der_gpcf}. $\hfill\Box$
\end{proof}

\noindent
{\it Proof of Theorem~\ref{optim_cond_4}} Let $x \in D \cap {\rm dom}\,f$. First, to prove the sufficiency, suppose that~\eqref{eq_optimal_cond_2} is fulfilled. Then there exist  nonnegative numbers $\lambda_k$, $\mu_{1,i}$, $\mu_{2,j}$, for $k \in \Theta(x)$, $i \in I(x)$, $j \in J(x)$, and an element $u^* \in ({\rm ker}\,A \cap {\rm ker}\,B)^{\perp}$ such that 
$\sum\limits_{k \in \Theta(x)}\lambda_k=1$ and
\begin{equation}\label{sum_0}
\sum\limits_{k \in \Theta(x)}\lambda_k v_k^*+\sum\limits_{i \in I(x)} \mu_{1,i} x_i^* + \sum\limits_{j \in J(x)}\mu_{2,j} u_j^* + u^*=0.
\end{equation}
For any $h \in T_{D}(x)$, we have $f'(x;h) \geq 0$. Indeed, if $h \notin T_{{\rm dom}\,f}(x)$, then $f'(x;h)=+\infty$ by \eqref{eq_dir_der_gpcf}. If $h \in T_{{\rm dom}\,f}(x)$, then 
\begin{equation}\label{h_inclusion}
h \in {\rm cone}\, ((D \cap {\rm dom}\, f) -x).
\end{equation}
To prove \eqref{h_inclusion}, we can argue as follows. Since $D$ and ${\rm dom}\, f$ are gpcs, thanks to \cite[Proposition~2.19]{Luan_Yao_Yen_2016}, we have $T_D(x)={\rm cone}(D-x)$ and $T_{{\rm dom}\,f}(x)={\rm cone}[({\rm dom}\,f)-x].$ This implies that $h \in {\rm cone}(D-x) \cap {\rm cone}[({\rm dom}\,f)-x]={\rm cone}\, ((D \cap {\rm dom}\, f) -x).$ So, \eqref{h_inclusion} is valid. To proceed furthermore, from 
\begin{equation*}
D \cap {\rm dom}\, f =\left\{ x \in X \mid Ax=y,\, Bx=z,\, \langle x^*_i, x \rangle \leq \alpha_i,\,  i \in I, \,\langle u^*_j, x \rangle \leq \gamma_j,\,  j \in J\right\},
\end{equation*}
we deduce that
\begin{equation}\label{rep_cone_D_cap_domf}
\begin{aligned}
{\rm cone}\, ((D \cap {\rm dom}\, f) -x)=\big\{ u \in X \mid & \; Au=0, \,  \langle x^*_i, u \rangle \leq 0,\,  i \in I(x), \\
&\; Bu=0, \, \langle u^*_j, u \rangle \leq 0,\,  j \in J(x) \big\}.
\end{aligned}
\end{equation}
Thus, by \eqref{h_inclusion} one has $\langle x^*_i, h \rangle \leq 0$ for every $i \in I(x)$, $\langle u^*_j, h \rangle \leq 0$ for all $j \in J(x)$, and $h \in {\rm ker}\,A \cap {\rm ker}\,B$. Since $h \in T_{{\rm dom}\,f}(x)$, one has $f'(x;h) =\max\{\langle v_k^*, h \rangle \mid  k \in \Theta(x)\}$ by Lemma~\ref{dir_der_gpcf}. Therefore, using the equality $\sum\limits_{k \in \Theta(x)}\lambda_k=1$ and \eqref{sum_0}, we have
\begin{equation*}
\begin{aligned}
f'(x;h) &\geq \sum\limits_{k \in \Theta(x)}\lambda_k \langle v_k^*, h \rangle\\
&= \left\langle -\left(\sum\limits_{i \in I(x)} \mu_{1,i} x_i^* + \sum\limits_{j \in J(x)}\mu_{2,j} u_j^* + u^*\right), h \right\rangle\\
&=\sum\limits_{i \in I(x)} \mu_{1,i} (-\left\langle x_i^*, h \right\rangle)+\sum\limits_{j \in J(x)} \mu_{2,j} (-\left\langle u_j^*, h \right\rangle)+\left\langle u^*, h \right\rangle \geq 0.
\end{aligned}
\end{equation*}
We have proved $f'(x;h) \geq 0$ for every $h \in T_{D}(x)$. Hence, $x$ is a solution of ${\rm (}{\mathcal{P}}{\rm )}$.

Now, to prove the necessity, denote the set on the right-hand side of \eqref{eq_optimal_cond_2} by $Q$ and suppose that $0 \notin Q$. We need to show that $x \notin {\rm Sol}{\rm (}{\mathcal{P}}{\rm )}$. Due to \cite[Theorem~2.7]{Luan_Yen_2015}, $Q$ is a nonempty gpcs. In particular, $Q$ is convex and weakly$^*$-closed. Since $0 \notin Q$, by the strong separation theorem \cite[Theorem~3.4(b)]{Rudin_1991} we can find $v \in X$ and $\gamma \in \mathbb{R}$ such that 
\begin{equation}\label{separation_Q_2}
\sup\limits_{x^* \in Q}\, \langle x^*,v \rangle < \gamma <\langle 0,v \rangle.
\end{equation}

On one hand, the first inequality in \eqref{separation_Q_2} implies that the linear functional $\langle \cdot,v\rangle $ is bounded from above on~$Q$. Hence, by \cite[Theorem~3.3]{Luan_Yen_2015}, the generalized linear programming problem
$\max\{\langle x^*,v \rangle \mid  x^* \in Q\}$ has a solution. Then we have  $\langle v^*, v \rangle \leq 0$ for all $v^* \in 0^+Q$ (see \cite[Proposition~3.5]{Luan_Yen_2015}). On the other hand, \eqref{eq_optimal_cond_2} yields
$$0^+Q={\rm cone}\,\{x^*_i\mid   i \in I(x)\}+{\rm cone}\,\{u^*_j \mid  j \in J(x)\}+ ({\rm ker}\,A \cap {\rm ker}\,B)^{\perp}.$$  
Therefore, $\langle x^*_i, v \rangle \leq 0$ for every $i\in I(x)$, $\langle u^*_j, v \rangle \leq 0$ for all $j \in J(x)$, and $v$ belongs to $(({\rm ker}\,A \cap {\rm ker}\,B)^{\perp})^{\perp}$. Since ${\rm ker}\,A \cap  {\rm ker}\,B$ is a closed linear subspace of $X$, applying \cite[Proposition~2.40]{Bonnans_Shapiro_2000}, one has $(({\rm ker}\,A \cap {\rm ker}\,B)^{\perp})^{\perp}={\rm ker}\,A \cap  {\rm ker}\,B.$ Consequently, formula \eqref{rep_cone_D_cap_domf} allows us to have $v \in {\rm cone}\, ((D \cap {\rm dom}\, f) -x)$, i.e., $v \in T_{D}(x) \cap T_{{\rm dom}\,f}(x)$. Hence, by Lemma~\ref{dir_der_gpcf} one has $f'(x;v)=\max\{\langle v_k^*, v \rangle \mid  k \in \Theta(x)\}$. 

For every $k \in \Theta(x)$, since $v_k^* \in Q$,  the inequalities in \eqref{separation_Q_2} yield $\langle v_k^*, v \rangle < \gamma <0$. Consequently, $f'(x;v)=\max\{\langle v_k^*, v \rangle \mid  k \in \Theta(x)\} < \gamma <0$. So, we have $x \notin {\rm Sol}{\rm (}{\mathcal{P}}{\rm )}$.

The proof is complete.
$\hfill\Box$

\section{Duality}
\setcounter{equation}{0}

In this final section, we will use the general conjugate duality scheme presented in \cite[pp.~107--108]{Bonnans_Shapiro_2000} to construct a dual problem for  ${\rm (}{\mathcal{P}}{\rm )}$ and obtain several duality theorems.

\medskip
If we define $F: X \to \bar{\mathbb{R}}$ and  $G: X \to X$, respectively, by $F(\cdot)=\delta(\cdot, D)$ and $G(x)=~x$, then problem ${\rm (}{\mathcal{P}}{\rm )}$ can be rewritten as
\begin{equation*}
{\rm (}{\widetilde{\mathcal{P}}}{\rm )} \qquad \min  \left\{ f(x)+F(G(x)) \mid x \in X\right\}. 
\end{equation*}
By the conjugate duality scheme in \cite[formulas~(2.298) and~(2.296)]{Bonnans_Shapiro_2000}, we obtain the following \textit{dual problem} of ${\rm (}{\widetilde{\mathcal{P}}}{\rm )}$:   
\begin{equation*}
{\rm (}{\widetilde{\mathcal{D}}}{\rm )} \qquad \max\limits \Big\{ \inf\limits_{x \in X} L(x, x^*) - F^*(x^*) \mid x^* \in X^*\Big\},
\end{equation*} where $L(x, x^*):=f(x)+\langle x^*, G(x) \rangle$ is the \textit{standard Lagrangian} of ${\rm (}{\widetilde{\mathcal{P}}}{\rm )}$. On one hand, it holds that $F^*(x^*)=\delta^*(\cdot, D) (x^*)$, where $\delta^*(\cdot, D) (x^*)=\sup\limits_{x \in D}\, \langle x^*, x \rangle$ is the \textit{support function} of $D$.  On the other hand,
\begin{equation*}
\begin{aligned}
\inf\limits_{x \in X} L(x, x^*) &= \inf\limits_{x \in X} \big( f(x)+\langle x^*, G(x) \rangle \big)= \inf\limits_{x \in X} \big( f(x)+\langle x^*, x \rangle \big)\\
&=-\sup\limits_{x \in X} \big( \langle -x^*, x \rangle -f(x) \big)=-f^*(-x^*).
\end{aligned}
\end{equation*}
Therefore, ${\rm (}{\widetilde{\mathcal{D}}}{\rm )}$  is nothing than the following problem
\begin{equation*}
{\rm (}{\mathcal{D}}{\rm )} \qquad \max\, \{g(x^*) \mid x^* \in X^*\}
\end{equation*} with $g(x^*):=\inf\limits_{x \in X} L(x, x^*) - F^*(x^*)=-f^*(-x^*)-\delta^*(\cdot, D) (x^*)$. 
Since $f$ and $\delta(\cdot, D)$ are proper generalized polyhedral convex functions, by \cite[Theorem~4.12]{Luan_Yao_Yen_2016} we can assert that $f^*$ and $\delta^*(\cdot, D)$ are proper generalized polyhedral convex functions. Hence, in particular, $x^*\mapsto f^*(-x^*)$ is a proper generalized polyhedral convex the function. If  $(-{\rm dom}\, f^*) \cap {\rm dom}\, \delta^*(\cdot, D)\neq\emptyset$, then $-g$ is a  proper generalized polyhedral convex function by \cite[Theorem~3.7]{Luan_Yao_Yen_2016}. So, the objective function of the maximization problem ${\rm (}{\mathcal{D}}{\rm )}$ is generalized polyhedral concave.  If $(-{\rm dom}\, f^*)\cap {\rm dom}\, \delta^*(\cdot, D)=\emptyset$, then $(-g)(x^*)=+\infty$ for all $x^* \in X^*$. In this case, the objective function of ${\rm (}{\mathcal{D}}{\rm )}$ is also generalized polyhedral concave. 

\medskip
A \textit{weak duality} relationship between ${\rm (}{\mathcal{P}}{\rm )}$ and ${\rm (}{\mathcal{D}}{\rm )}$ can be described as follows.
\begin{theorem}\label{Weak_duality_thm} {\rm (Weak duality theorem)} For every $u \in D$ and $u^* \in X^*$, the inequality $ g(u^*)\leq f(u)$ holds. Hence, if $f(u) = g(u^*)$, then $u \in   {\rm Sol}{\rm (}{\mathcal{P}}{\rm )}$ and $u^* \in {\rm Sol}{\rm (}{\mathcal{D}}{\rm )}$.
\end{theorem}
\begin{proof} Given any $u \in D$ and $u^* \in X^*$, it suffices to observe that
\begin{equation}\label{rep_g}
\begin{aligned}
g(u^*) & = -f^*(-u^*)-\delta^*(\cdot, D) (u^*)\\
& = \inf\limits_{x \in X} \big[\langle u^*, x \rangle + f(x)\big] - \sup\limits_{x \in D} \langle u^*, x \rangle \\
&\leq \langle u^*, u \rangle + f(u) - \langle u^*, u \rangle =f(u).
\end{aligned}
\end{equation}
This justifies the assertions of the theorem. $\hfill\Box$
\end{proof}

Since the existence of an element $u^*$ satisfying $u^* \in N_D(u) \cap (-\partial f(u))$ is equivalent to the property $0\in \partial f(u)+N_D(u)$, the next statement can be interpreted as a sufficient optimality condition for ${\rm (}{\mathcal{P}}{\rm )}$  and ${\rm (}{\mathcal{D}}{\rm )}$.
\begin{proposition}\label{Duality_prop} If $u \in X$ and $u^* \in N_D(u) \cap (-\partial f(u))$, then one has $u \in {\rm Sol}{\rm (}{\mathcal{P}}{\rm )}$ and $u^* \in {\rm Sol}{\rm (}{\mathcal{D}}{\rm )}$. Moreover, the optimal values of ${\rm (}{\mathcal{P}}{\rm )}$ and ${\rm (}{\mathcal{D}}{\rm )}$ are equal. 
\end{proposition}
\begin{proof}
Suppose that $u \in X$ and $u^* \in N_D(u) \cap (-\partial f(u))$. Then $u \in D \cap {\rm dom}\,f$. On one hand, since $-u^* \in \partial f(u)$, by \cite[Theorem~2.4.2(iii)]{Zalinescu_2002} we can assert that $$f(u)+f^*(-u^*)=\langle-u^*, u \rangle.$$ So, $-f^*(-u^*)=f(u)+\langle u^*, u \rangle.$ On the other hand, the inclusion $u^* \in N_D(u)$ implies that $\sup\{\langle u^*, x \rangle \mid x \in D\}=\langle u^*, u \rangle;$ hence $\delta^*(\cdot, D) (u^*)=\langle u^*, u \rangle$. Consequently, $$g(u^*)= -f^*(-u^*)-\delta^*(\cdot, D) (u^*)=f(u).$$ Thus, the desired conclusions follow from Theorem \ref{Weak_duality_thm}. $\hfill\Box$
\end{proof}	

If the optimal value of ${\rm (}{\mathcal{D}}{\rm )}$ equals to the optimal value of ${\rm (}{\mathcal{P}}{\rm )}$, then one says that the \textit{strong duality} relationship among the dual pair holds.  We are going to show that if either $f$ is polyhedral convex or $D$ is polyhedral convex, then this property is available under a mild condition.
\begin{theorem}\label{Strong_dual_thm_1} {\rm (Strong duality theorem I)} Assume that either $f$ is a proper polyhedral convex function and $D$ is a nonempty generalized polyhedral convex set, or $f$ is a proper generalized polyhedral convex function and $D$ is a nonempty polyhedral convex set. If one of the two problems has a solution, then both of them have solutions and the optimal values are equal. 
\end{theorem}
\begin{proof} Under the assumptions of the theorem, we suppose firstly that ${\rm (}{\mathcal{P}}{\rm )}$ has a solution $u$. Then, according to Theorem \ref{Optimality_condition_2_pc}, it holds that $0 \in \partial f(u) + N_D(u)$. Hence there exists $u^* \in N_D(u) \cap (-\partial f(u))$. Applying Proposition \ref{Duality_prop} yields the solution existence of ${\rm (}{\mathcal{D}}{\rm )}$ and the equality of the optimal values.

Secondly, suppose that ${\rm (}{\mathcal{D}}{\rm )}$ has a solution $u^*$. Since $f$ is a proper gpcf, ${\rm dom}\,f$ is a nonempty gpcs by \cite[Theorem~3.2]{Luan_Yao_Yen_2016}. If $f$ is a proper pcf then, also by \cite[Theorem~3.2]{Luan_Yao_Yen_2016},  ${\rm dom}\,f$ is a nonempty pcs. Thus, by the assumptions of the theorem, ${\rm dom\,}f$ and $-D$ are gpcs, and one of them is polyhedral convex. Hence, in accordance with Proposition~2.11 from \cite{Luan_Yao_Yen_2016}, the set $({\rm dom}\,f)-D$ is polyhedral convex. In particular, $({\rm dom}\,f)-D$ is a closed set. Let us show that $D \cap {\rm dom}\,f$ is nonempty. On the contrary, suppose that $D \cap {\rm dom}\,f=\emptyset$. Hence, $0 \notin ({\rm dom}\, f)-D$. Since the nonempty set $({\rm dom}\, f)-D$ is closed, by the strong separation theorem \cite[Theorem~3.4(b)]{Rudin_1991} there exist $x^* \in X^*$ and real number~$\varepsilon$ such that 
$0 < \varepsilon < \langle x^*, x-u \rangle$ for all $x \in {\rm dom}\,f$ and $u \in D$. 
Consequently, 
\begin{equation}\label{diff_1}
\varepsilon + \sup\limits_{u \in D}\langle x^*, u \rangle \leq   \inf\limits_{x \in {\rm dom}\,f}\langle x^*, x \rangle.
\end{equation}
On one hand, for any $\lambda >0$, using the equalities in \eqref{rep_g} and the inequality \eqref{diff_1} we have
\begin{equation*}
\begin{aligned}
&g(u^*+\lambda x^*) \\
& =\inf\limits_{x \in X}[\langle u^*+\lambda x^*, x \rangle+f(x)]- \sup\limits_{x \in D}\langle u^*+\lambda x^*, x \rangle\\
&=\inf\limits_{x \in {\rm dom}\,f}[\langle u^*+\lambda x^*, x \rangle+f(x)]- \sup\limits_{u \in D}\langle u^*+\lambda x^*, u \rangle\\
&\geq \inf\limits_{x \in {\rm dom}\,f}[\langle u^*, x \rangle+f(x)] + \lambda \inf\limits_{x \in {\rm dom}\,f}\langle x^*, x \rangle 
-\sup\limits_{u \in D}\langle u^*, u \rangle - \lambda \sup\limits_{u \in D}\langle x^*, u \rangle\\
&=\left(\inf\limits_{x \in {\rm dom}\,f}[\langle u^*, x \rangle+f(x)] - \sup\limits_{u \in D}\langle u^*, u \rangle\right) + \lambda\left[\inf\limits_{x \in {\rm dom}\,f}\langle x^*, x \rangle -\sup\limits_{u \in D}\langle x^*, u \rangle\right]\\
&\geq g(u^*)+ \lambda \varepsilon.
\end{aligned}
\end{equation*} 
On the other hand, since $u^*$ is a solution of ${\rm (}{\mathcal{D}}{\rm )}$, the estimate $g(u^*) \geq g(u^*+\lambda x^*)$ is valid. Hence, $g(u^*) \geq  g(u^*)+ \lambda \varepsilon$. This contradicts the fact that $\lambda, \varepsilon$ are positive numbers. Thus, we have proved that $({\rm dom}\,f) \cap D \neq \emptyset$. Setting $\gamma=g(u^*)$ and applying Theorem~\ref{Weak_duality_thm}, we obtain $f(x) \geq\gamma$ for all $x\in D$. Therefore, on account of  Theorem \ref{FW_exist_thm}, we can assert that ${\rm (}{\mathcal{P}}{\rm )}$ has a solution. Finally, to show that the optimal values of ${\rm (}{\mathcal{P}}{\rm )}$ and ${\rm (}{\mathcal{D}}{\rm )}$ are equal, it suffices to use the result already obtained in the first part of this proof.  $\hfill\Box$
\end{proof}

\begin{example} Consider problem ${\rm (}{\mathcal{P}}{\rm )}$ in the setting and notations of Example \ref{Example1_gpcp}.  To have a concrete form of the dual problem ${\rm (}{\mathcal{D}}{\rm )}$, we have to find the function $g$. Suppose that $x^* \in X^*$ and $|g(x^*)| < \infty$. Since $f$ is proper, for $\alpha:=\inf\limits_{x \in X}\, [f(x)+\langle x^*, x \rangle]$, we have $\alpha<+\infty$. In addition, as $D$ is nonempty,  the number $\beta:=\sup\limits_{x \in D}\,\langle x^*, x \rangle$ is greater than~$-\infty$. Thus, the equalities in \eqref{rep_g} yield $+\infty>\alpha=g(x^*)+\beta>-\infty.$ Hence, both~$\alpha$ and $\beta$ are finite. In particular, the function $x\mapsto f(x)+\langle x^*, x \rangle$ is bounded from below on $X$. Since $f(x) = \max\{\langle v^*_1,x\rangle+1, \langle v^*_2,x\rangle \}$, one see that
		\begin{equation}\label{rep_sum_f_xstar}
		f(\cdot)+\langle x^*, \cdot \rangle = \max\left\{\langle v^*_1+x^*, \cdot \rangle+1, \langle v^*_2+x^*, \cdot\rangle\right\}
		\end{equation}
is a polyhedral convex function. So, according to Theorem \ref{FW_exist_thm}, the generalized polyhedral convex optimization problem $\min \left\{ f(x) + \langle x^*, x \rangle \mid x \in X\right\}$ has a solution. Therefore, by \eqref{rep_sum_f_xstar} and Corollary~\ref{solution_existence_unconstrained}, we must have $0 \in {\rm conv}\,\{v_1^*+x^*, v^*_2+x^*\}$. Let $\lambda_1\geq~0,$ $\lambda_2\geq 0$ be such that $\lambda_1+\lambda_2=1$ and $\lambda_1(v_1^*+x^*)+\lambda_2(v_2^*+x^*)=0$. It is clear that 
		\begin{equation*}
		\begin{aligned}
		x^*&=-\lambda_1v_1^*-\lambda_2v_2^*=-\lambda_1(x^*_1-x^*_2)-\lambda_2(-x^*_1-x^*_2)\\
		&=(1-2\lambda_1)x^*_1+x^*_2.
		\end{aligned}
		\end{equation*}
Writing $\lambda=1-2\lambda_1$, we obtain $x^*=\lambda x^*_1 + x^*_2$ with $\lambda \in [-1,1]$. It is a simple matter to verify that $e_0+t_1e_1\in D$ for all $t_1 \leq 1$. Hence, $$\sup\limits_{x \in D}\,\langle x^*, x \rangle  \geq \sup\limits_{t_1 \leq 1 }\, \langle x^*, e_0+t_1e_1 \rangle  = \sup\limits_{t_1 \leq 1 }  \lambda t_1.$$ If $\lambda < 0$, then $\sup\limits_{t_1 \leq 1 }  \lambda t_1 = + \infty.$ So, we get $\sup\limits_{x \in D}\, \langle x^*, x \rangle=+\infty$, which contradicts the fact that $\beta\in\mathbb R$. Thus we have proved that if $g(x^*)$ is finite, then there exists $\lambda \in [0,1]$ satisfying $x^*=\lambda x^*_1 + x^*_2$. Let us compute the value  $g(x^*)$ when $|g(x^*)| < \infty$. Suppose that  $x^*=\lambda x^*_1 + x^*_2$ with $\lambda \in [0,1]$. For every $x \in D$, since $\langle x^*_1, x \rangle \leq 1$ and $\langle x^*_2, x \rangle \leq 2$, one has $\langle x^*, x \rangle  = \lambda \langle x^*_1, x \rangle + \langle x^*_2, x \rangle \leq \lambda+2$. In addition, as $e_0+e_1+2e_2 \in D$ with $\langle x^*, e_0+e_1+2e_2  \rangle= \lambda+2$, we must have  $\sup\limits_{x \in D}\,\langle x^*, x \rangle=\lambda+2$. From \eqref{rep_sum_f_xstar} we see that
		\begin{equation*}\label{rep_sum_f_xstar_2}
		f(x)+\langle x^*, x \rangle = \max\left\{(\lambda +1)\langle x^*_1,x \rangle +1,  (\lambda -1)\langle x^*_1,x \rangle\right\}.
		\end{equation*}
If $\langle x^*_1,x \rangle  \geq -\frac{1}{2}$, then 
		\begin{equation*}
		\begin{aligned}
		f(x)+\langle x^*, x \rangle &=  (\lambda +1)\langle x^*_1,x \rangle +1  \\
		&\geq  (\lambda +1) \left(-\frac{1}{2}\right)+1=\frac{1}{2}-\frac{\lambda}{2}.
		\end{aligned}
		\end{equation*}
If $\langle x^*_1,x \rangle < -\frac{1}{2}$, then 
		\begin{equation*}
		\begin{aligned}
		f(x)+\langle x^*, x \rangle &=  (\lambda -1)\langle x^*_1,x \rangle  \\
		&\geq  (\lambda -1) \left(-\frac{1}{2}\right)=\frac{1}{2}-\frac{\lambda}{2}.
		\end{aligned}
		\end{equation*}
Since $\langle x^*_1, - \frac{1}{2}e_1 \rangle = -\frac{1}{2}$, we have $f(- \frac{1}{2}e_1)+\langle x^*, - \frac{1}{2}e_1 \rangle=\frac{1}{2}-\frac{\lambda}{2}.$ Consequently, $\inf\limits_{x \in X}\, [f(x)+\langle x^*, x \rangle]=\frac{1}{2}-\frac{\lambda}{2}$. Thus, the equalities in \eqref{rep_g} imply that 
		\begin{equation*}
		g(x^*)=\begin{cases}
		-\frac{3}{2}-\frac{3}{2}\lambda &\text{ if } x^*=\lambda x^*_1+x^*_2\, \text{ with } 0 \leq \lambda \leq 1  \\
		-\infty &\text{ otherwise}. \\
		\end{cases} 
		\end{equation*}  
Using this formula for $g$, it is easy to check that $x^*_2$ is a unique solution of ${\rm (}{\mathcal{D}}{\rm )}$ with $g(x^*_2)=-\frac{3}{2}$. In Example \ref{solve_Example1_gpcp}, we have shown that ${\rm (}{\mathcal{P}}{\rm )}$ has a nonempty solution set and the optimal value is $-\frac{3}{2}$. These facts justify the assertion of Theorem \ref{Strong_dual_thm_1} for the problems ${\rm (}{\mathcal{P}}{\rm )}$ and ${\rm (}{\mathcal{D}}{\rm )}$ which we are dealing with.
\end{example}

The conclusion of Theorem~\ref{Strong_dual_thm_1} may not true in the general case, where one just assumes that $f$ is a proper generalized polyhedral convex function and $D$ is a nonempty generalized polyhedral convex set. 
\begin{example} Consider problem ${\rm (}{\mathcal{P}}{\rm )}$ in the setting and notations of Example~\ref{Ex_optim_cond}. We know that ${\rm (}{\mathcal{P}}{\rm )}$ has a unique solution $x=0$. Recall that ${\rm dom}\,f=D_1$ and $f(x)=\langle v^*, x \rangle$ for all $x \in D_1$. In addition, since $D_1$ is the orthogonal complement of $X_1$, we have
		\begin{equation}\label{Example2_rep_g_1}
		\inf\limits_{x \in X}\, [f(x)+\langle x^*, x \rangle]= \inf\limits_{x \in D_1}\, \langle v^*+x^*, x \rangle = \begin{cases}
		0 & \text{ if } x^* \in -v^* + X_1\\
		-\infty & \text{ if } x^* \notin -v^* + X_1.
		\end{cases}
		\end{equation}
Similarly, since $D$ is the orthogonal complement of $X_2$, 
		\begin{equation}\label{Example2_rep_g_2}
		\sup\limits_{x \in D}\,\langle x^*, x \rangle= \begin{cases}
		0 & \text{ if } x^* \in X_2\\
		+\infty & \text{ if } x^* \notin X_2.
		\end{cases}
		\end{equation}
Combining \eqref{Example2_rep_g_1}, \eqref{Example2_rep_g_2} with the equalities in \eqref{rep_g} yields
		\begin{equation*}
		g(x^*)= \begin{cases}
		0 & \text{ if } x^* \in (-v^* + X_1) \cap X_2\\
		-\infty & \text{ otherwise}.
		\end{cases}
		\end{equation*}
Since $(-v^* + X_1) \cap X_2 = \emptyset$ (see Example~\ref{Ex_optim_cond}), we can assert that $g(x^*)=-\infty$ for all $x^* \in X^*$. Therefore, ${\rm (}{\mathcal{D}}{\rm )}$ has no solution. Thus, it happens that  ${\rm (}{\mathcal{P}}{\rm )}$ has a solution, while ${\rm (}{\mathcal{D}}{\rm )}$  has an empty solution set.
\end{example}

The assumption of Theorem~\ref{Strong_dual_thm_1} implies that $D-{\rm dom}\,f$ is a polyhedral convex set in $X$. In particular, $D-{\rm dom}\,f$ is closed. Interestingly, in a Banach space setting, the polyhedral convexity of $D-{\rm dom}\,f$ can be replaced by its closedness -- a weaker property. 

\begin{theorem}\label{Strong_dual_thm_2} {\rm (Strong duality theorem II)} Suppose that $X$ is a Banach space and the set $D-{\rm dom}\,f$ is closed. If one of the two problems ${\rm (}{\mathcal{P}}{\rm )}$ and ${\rm (}{\mathcal{D}}{\rm )}$ has a solution, then both of them have solutions and the optimal values are equal. 
\end{theorem}
\begin{proof} First, suppose that ${\rm (}{\mathcal{P}}{\rm )}$ has a solution $u$. Then, by the closedness of $D-{\rm dom}\,f$ and Theorem~\ref{Optimality_condition_3_Banach}, we have $0 \in \partial f(u) + N_D(u)$. Select any $u^* \in N_D(u) \cap (-\partial f(u))$. By Proposition~\ref{Duality_prop}, $u^*$ is a solution of ${\rm (}{\mathcal{D}}{\rm )}$. Moreover, the optimal values of  ${\rm (}{\mathcal{P}}{\rm )}$ and  ${\rm (}{\mathcal{D}}{\rm )}$ are equal.

Now, suppose that ${\rm (}{\mathcal{D}}{\rm )}$ has a solution $u^*$. Arguing similarly as in the proof of Theorem~\ref{Strong_dual_thm_1} (the closedness of $D-{\rm dom}\,f$ allows us to apply the strong separation theorem), we can prove that ${\rm (}{\mathcal{P}}{\rm )}$ has a solution and the optimal values of ${\rm (}{\mathcal{P}}{\rm )}$ and ${\rm (}{\mathcal{D}}{\rm )}$ are equal. $\hfill\Box$ 
\end{proof}

In optimization theory, a strong duality theorem can be formulated as a combined statement about the solution existence of the primal and dual problems when they have feasible points where the objective functions are finite, and the equality of the optimal values. In that spirit, for generalized polyhedral convex optimization problems we have the next result.

\begin{theorem}\label{Strong_dual_thm_3} {\rm (Strong duality theorem III)} Suppose that the problems ${\rm (}{\mathcal{P}}{\rm )}$ and ${\rm (}{\mathcal{D}}{\rm )}$ have feasible points, at which the values of the object functions are finite. Then both problems have solutions. In addition, if either $f$ or $D$ is polyhedral convex, then there is no duality gap between the problems.
\end{theorem}
\begin{proof} Let $u \in D$ and $u^* \in X^*$ be such that $f(u)$ and $g(u^*)$ are finite. By Theorem~\ref{Weak_duality_thm}, we have $f(x) \geq g(u^*)$ for every $x \in D$. Thus, $f$ is bounded from below on $D$ and $D \cap {\rm dom}\,f \neq \emptyset$. Therefore, ${\rm (}{\mathcal{P}}{\rm )}$ has a solution by Theorem \ref{FW_exist_thm}. To show that ${\rm (}{\mathcal{D}}{\rm )}$  possesses a solution, we first observe by Theorem~\ref{Weak_duality_thm} that $-g(x^*) \geq - f(u)$ for all $x^* \in X^*$. Hence, the proper generalized polyhedral convex function $(-g)$ is bounded from below  on $X^*$ by the finite value $(-f(u))$. Consequently, by Theorem~\ref{FW_exist_thm}, the problem $\min\,\{-g(x^*) \mid x^* \in X^*\}$ has a solution. Since ${\rm (}{\mathcal{D}}{\rm )}$ is equivalent to the latter, the solution set of ${\rm (}{\mathcal{D}}{\rm )}$ is nonempty.   

Now, if either $f$ or $D$ is polyhedral convex, then by using Theorem~\ref{Strong_dual_thm_1} we can assert that the optimal values of ${\rm (}{\mathcal{P}}{\rm )}$ and ${\rm (}{\mathcal{D}}{\rm )}$ are equal. $\hfill\Box$
\end{proof}	

Concerning Theorem~\ref{Strong_dual_thm_3}, the following question seems to be interesting: \textit{Whether the conclusion ``there is no duality gap between two problems'' is still true, if one drops the assumption ``either $f$ or $D$ is polyhedral convex''?} Our attempts in constructing a counterexample have not achieved the goal, so far.

\begin{acknowledgements}
The authors would like to thank Professor Nguyen Dong Yen for valuable discussions on the subject.

\end{acknowledgements}

\end{document}